\documentclass[12pt]{article}
\usepackage{amsfonts, amssymb, amsmath, amsthm, amssymb}
\usepackage{esint}
\usepackage{pifont}
\usepackage{bbding}
\usepackage{latexsym}
\usepackage{mathrsfs}
\usepackage{graphicx}
\usepackage{lineno}
\usepackage[colorlinks, linkcolor=blue, anchorcolor=blue, citecolor=blue]{hyperref}
\usepackage{verbatim}
\allowdisplaybreaks %allow math displays to break into two pages
\usepackage[left=1in, right=1in, bottom=1.4in]{geometry}
\newtheorem{theorem}{\indent Theorem}[section]

\newtheorem{lemma}{\indent Lemma}[section]
\newtheorem{remark}{\indent Remark}[section]
\numberwithin{equation}{section}

\newcommand{\mb}{\mathbb}

\newcommand{\R}{\mathbb{R}}

\newcommand{\va}{\varepsilon}

\newcommand{\de}{\delta}

\newcommand{\na}{\nabla}

\newcommand{\e}{\varepsilon}

\numberwithin{equation}{section}

\date{} %\linespread{2.0}
\begin{document}

\title{Quantitative Estimates in Reiterated Homogenization}

\author{Weisheng Niu\thanks{Supported by the NSF of China (11971031, 11701002).}
, Zhongwei Shen\thanks{Supported in part by NSF grant DMS-1856235.}, 
and Yao Xu \thanks{Supported by China Postdoctoral Science Foundation (2019TQ0339).}
}
\maketitle
\pagestyle{plain}

\begin{abstract}
This paper investigates quantitative estimates in the homogenization of second-order elliptic systems 
with periodic coefficients that oscillate on multiple  separated scales. 
We establish large-scale interior and boundary Lipschitz estimates down to the finest microscopic scale
via iteration and rescaling arguments.
We also obtain  a convergence rate  in the $L^2$ space by the reiterated homogenization method.   

\medskip


\medskip


\end{abstract}

%%%%%%%%%%%%%%%%%%%%%%%%%%%%%%%%%%%%%%%%%%%%%%%%%%%%%%%%%%%%%%%%%%%%%%%%%%%%%%%%%%%%%%%%%%%%%%%%%%%

\section{Introduction}\label{section-1}

In this paper we  investigate quantitative estimates in the homogenization of  elliptic systems with
periodic coefficients that oscillate on multiple separated scales.
More precisely,   consider the $m\times m$ elliptic system in divergence form,
\begin{equation}\label{eq11}
\mathcal{L}_\varepsilon (u_\varepsilon) =F
\end{equation}
in a bounded domain $\Omega\subset \mathbb{R}^d$  $(d\ge 2)$,
where  
\begin{equation}\label{operator}
 \mathcal{L}_\varepsilon= -\text{\rm div} \big( A^\e (x)\nabla \big)=
-\text{\rm div} \big( A(x, x/\e_1 , x/\e_2, \dots, x/\e_n \big)\nabla \big),
 \end{equation}
and $\{ 0< \e_n <\e_{n-1}<\cdots< \e_1< 1 \}$  represents a set of $n$ ordered lengthscales, all depending on a single parameter $\e$.
We assume that the coefficient tensor $A=A(x, y_1, y_2, \dots, y_n)$ is real, bounded measurable, and satisfies the ellipticity condition,
\begin{align}
\|A\|_{L^\infty( \R^{d\times (n+1)})}\leq \frac{1}{\mu} \quad \text{ and } \quad
 \mu |\xi|^2\le \langle A \xi, \xi \rangle \label{elcon}
\end{align}
for any $\xi \in \mathbb{R}^{m\times d}$, where $\mu>0$, and the periodicity condition 
\begin{align}
A(x,y_1+z_1, \cdot\cdot\cdot, y_n+z_n)=A(x,y_1,\cdot\cdot\cdot,y_n) \quad\text{for any } (z_1, \cdots, z_n)\in  \mathbb{Z}^{d\times n}  \label{pcon}.
\end{align}
We also impose the H\"older continuity condition on $A$: there exist constants $L\ge 0$ and $ 0<\theta \leq 1$ such that
\begin{align}\label{lipcon}
|A(x,y_1,\cdot\cdot\cdot,y_{n-1},y_n)-A(x',y_1',\cdot\cdot\cdot,y'_{n-1},y_n)|\leq L \Big\{|x-x'|+\sum_{\ell =1}^{n-1}
|y_\ell -y'_\ell | \Big\}^\theta
\end{align}
for  $ x,x', y_1, \dots, y_n, y_1^\prime, \dots, y_{n-1}^\prime\in \mathbb{R}^d$.
Note that no continuity  condition is needed for the last variable $y_n$.

Homogenization problems with multiscale structures were first considered  in the 1930s
 by Bruggeman \cite{b1935}.
 In the case where  $\e_k=\e^k$  for  $1\le k\le n$, 
the qualitative homogenization  theory for $\mathcal{L}_\e$ in (\ref{operator}) was  established  in the 1970s 
by Bensoussan, Lions, and Papanicolaou \cite{lions1978}. 
Let $u_\e$ be a weak solution of the Dirichlet problem,
\begin{equation}\label{DP}
\mathcal{L}_\e (u_\e) =F \quad \text{ in } \Omega
\quad \text{ and } \quad
u_\e =f \quad \text{ on } \partial\Omega.
\end{equation}
Assume that $A$ satisfies (\ref{elcon})-(\ref{pcon})  and some continuity  condition.
It is known  that  $u_\e$ converges weakly in $H^1(\Omega)$ to the  solution $u_0$ of 
the homogenized problem,
\begin{equation}\label{DP-0}
\mathcal{L}_0 (u_0) =F \quad \text{ in } \Omega
\quad \text{ and } \quad
u_0 =f \quad \text{ on } \partial\Omega,
\end{equation}
where $\mathcal{L}_0 =-\text{\rm div} \big( \widehat{A} (x)\nabla \big)$ is a second-order elliptic operator.
The effective tensor $\widehat{A}(x)$ is obtained by homogenizing separately and successively the different scales,
starting from the finest one $\e_n $, as follows.
One fixes $(x, y_1, \dots, y_{n-1})$
and homogenizes  the last variable $y_n=x/\e_n $
 in $A_n=A(x, y_1, \dots, y_n)$
 to obtain $A_{n-1} (x, y_1, \dots, y_{n-1})$.
 Repeat the same procedure  on $A_{n-1}$  to obtain $A_{n-2}$, and continue until one  arrives at  $A_0(x)$, which is $\widehat{A}(x)$.
This process, in which at each step the  standard homogenization  is performed  on  an operator with a parameter,
 is referred in \cite{lions1978} as reiterated homogenization.
 For more recent work in the reiterated homogenization theory and its applications,
we refer the reader to \cite{Allaire1996Multiscale,lions2000,lions2001,luk2002,nonrei2005,past07, luk2008, past2013, past2016} and their references.
In particular, using the  method of multiscale convergence,
 Allaire and Briane \cite{Allaire1996Multiscale} obtained qualitative results for $\mathcal{L}_\e$
 in  a general case
 %, where 
 %\begin{equation}\label{g-A}
 %A^\e(x) =A(x, x/\e_1, \dots, x/\e_n),
 %\end{equation}
  under the condition  of separation of scales,
  \begin{equation}\label{s-condition}
  \e_1\to 0 \quad \text{  and  } \quad 
  \e_{k+1}/\e_k \to 0\quad  \text{  for } 1\le k\le n-1 ,  \text{ as } \e\to 0.
  \end{equation}

This paper is devoted to the quantitative homogenization theory for the operator $\mathcal{L}_\e$
and concerns  problems of convergence rates and large-scale regularity estimates.
We point out that in the case $n=1$, where $A^\e (x)=A(x/\e)$ or $A(x, x/\e)$,
major progress has been made in quantitative homogenization   in recent years.
We refer the reader to
\cite{al87, suslinaD2013,  klsj2013, klsc2014, armstrongcpam2016, shenan2017, nsx,  shennote2017} and their references 
 for  the periodic case,  and to 
\cite{Gloria2015, armstrongan2016,  armstrongar2016, Gloria2017, armstrong-book} and their references 
for quantitative homogenization 
in the stochastic setting.
The primary purpose of this paper is to extend  quantitative estimates in periodic homogenization 
for  $n=1$ to the case $n>1$, where the operator $\mathcal{L}_\e$ is  used to model a composite medium 
with several microscopic scales.

Our main results are given   in the following two theorems. We
establish the large-scale interior and boundary Lipschitz estimates
down to the finest scale $\e_n$, assuming that  the scales $0<\e_n<\e_{n-1}
< \cdots< \e_1< \e_0=1$ are well-separated in the  sense that 
there exists a positive integer $N$  such that
\begin{equation}\label{w-s-cond}
\left( \frac{\e_{k+1}}{\e_k} \right)^N \le   \frac{\e_k}{\e_{k-1}}    \quad \text{ for } 1\le k\le n-1.
\end{equation} 
In particular, this includes the case where $\e_k=\e^{\lambda_k}$
with 
$\lambda_0 =0  < \lambda_1< \lambda_2< \dots < \lambda_n<\infty$ and $0<\e\le 1$,
but excludes the case $(\e_1, \e_2)=(\e, \e ( |\log \e | +1)^{-1} )$.

\begin{theorem}\label{lipth}
Suppose that $A$ satisfies  conditions  \eqref{elcon},
\eqref{pcon}, and  \eqref{lipcon} for some $0<\theta \le 1$. 
Also assume that $0<\e_n<\e_{n-1}<\dots< \e_1 <\e_0=1$ and  (\ref{w-s-cond}) holds.
For $B_R=B(x_0, R)$ with $0<\e_n  < R\le 1$,
let  $u_\varepsilon\in H^1(B_R; \mathbb{R}^m)$ be a weak solution of $\mathcal{L}_\varepsilon ( u_\varepsilon) =F$ in $B_R$,
 where $F\in L^p(B_R; \mathbb{R}^m)$ for some $p>d$.
  Then for $0<\e_n \leq r< R$,
\begin{align}\label{relipth}
\left(\fint_{B_r} |\na u_\varepsilon|^2\right)^{1/2}\leq C \left\{  \left(\fint_{B_R}
|  \nabla u_\varepsilon|^2\right)^{1/2} +
R\left(\fint_{B_R} |F  |^p\right)^{1/p}\right\},
\end{align}
where $C$ depends at most on $d$, $n$, $m$, $\mu$, $p$, $(\theta, L)$ in \eqref{lipcon}, and $N$ in (\ref{w-s-cond}).
\end{theorem}

Let $\Omega$ be a bounded domain in $\mathbb{R}^d$.
Define $D_r=D(x_0, r)=B(x_0, r)\cap \Omega$ and
$\Delta_r=\Delta (x_0, r)=B(x_0, r) \cap \partial\Omega$, where $x_0\in \partial\Omega$ and $0<r<\text{\rm diam} (\Omega)$.

\begin{theorem}\label{blipth}
Assume that $A$ and $(\e_1, \e_2, \dots, \e_n)$ satisfy the same conditions as in Theorem \ref{lipth}.
Let $\Omega$ be a bounded $C^{1, \alpha}$ domain in $\mathbb{R}^d$ for some $\alpha>0$.
Let $u_\varepsilon\in H^1(D_R; \mathbb{R}^m)$ be a weak solution to 
$\mathcal{L}_\varepsilon (u_\varepsilon)=F$ in $D_R$ and $u_\va=f$ on $\Delta_R$, 
where $\e_n  < R\leq 1$,
$F\in L^p(D_R; \mathbb{R}^m)$ for some $p>d$, and $f\in C^{1,\nu}(\Delta_R)$ for some $0<\nu\le \alpha$.
 Then for $0<\e_n  \leq r< R$, 
\begin{align}\label{reblipth}
\left(\fint_{D_r} |\na u_\varepsilon|^2\right)^{1/2}\leq   C \left\{ \left(\fint_{D_R}
| \nabla  u_\varepsilon|^2\right)^{1/2} +
R\left(\fint_{D_R} |F |^p\right)^{1/p}+  R^{-1}   \|f \|_{{C}^{1,\nu}(\Delta_R)}\right\},
\end{align}
where
$C$ depends at most on $d$, $m$, $n$,  $\mu$, $p$, $\nu$,
$(\theta, L)$ in \eqref{lipcon}, $N$ in  (\ref{w-s-cond}), and $\Omega$.
\end{theorem}

\begin{remark}
{\rm
Under the additional assumption that $A=A(x, y_1, \dots, y_n)$ is also H\"older continuous in $y_n$, 
estimates (\ref{relipth}) and (\ref{reblipth}) imply the uniform pointwise interior and boundary Lipschitz estimates for $u_\e$,
respectively.
To see this, one introduces a dummy variable $y_{n+1}$
and considers  the tensor $\widetilde{A}(x, y_1,\dots, y_n, y_{n+1})=A(x, y_1, \dots, y_n)$.
Since $\e_{n+1}$ may be arbitrarily small,
it follows that the inequalities (\ref{relipth}) and (\ref{reblipth}) hold
for any $0<r<R\le 1$.
By letting $r\to0$ we see that $|\nabla u_\e (x_0)|$ is bounded by the right-hand sides of the inequalities.
}
\end{remark}

\begin{remark}
{\rm
In the case $A^\e(x) =A(x/\e)$, Theorems \ref{lipth} and \ref{blipth} were proved by Avellaneda and Lin 
in a seminal paper \cite{al87} by using a compactness method.
The boundary Lipschitz estimate  in Theorem \ref{blipth} was extended in \cite{klsj2013} to solutions with Neumann conditions.
Also see \cite{armstrongcpam2016} for operators with almost-periodic coefficients 
and \cite{armstrongan2016, armstrongar2016} for large-scale Lipschitz estimates in stochastic homogenization.
Our results for $n>1$ are new even in the case $A^\e (x)=A(x/\e, x/\e^2)$.
}
\end{remark}

We now describe our approach to  the proof of Theorem \ref{lipth};
the same approach works equally well for Theorem \ref{blipth}.
The proof  is divided into two steps.
In the first step we prove the estimate (\ref{relipth}) for the 
case $\e_1 \le  r<R\le 1$.
To do this, we use a general approach developed in \cite{armstrongan2016} 
by Armstrong and Smart (also see \cite{ armstrongcpam2016,armstrongar2016}),
which reduces the large-scale Lipschitz estimates to a problem of approximating solutions  of
$\mathcal{L}_\e (u_\e)=F$ 
by solutions of $\mathcal{L}_0 (u_0)=F$ in the $L^2$ norm.
Given $u_\e$, to find a good approximation $u_0$, we use the idea of reiterated homogenization
and introduce a (finite) sequence of approximations  as follows.
One  first approximates $u_\e$ by solutions of $-\text{\rm div} \big( A_{n-1}^\e (x)\nabla u_{\e, n-1}\big)=F$,
where $A_{n-1}^\e (x) =A_{n-1} (x, x/\e_1, \dots, x/\e_{n-1} )$
and $A_{n-1} (x, y_1, \dots, y_{n-1})$ is the effective tensor  for $A_n=A(x, y_1, \dots, y_{n-1}, y_n)$,
with $(x, y_1, \dots, y_{n-1})$ fixed as parameters.
The function $u_{\e, n-1}$ is then approximated by a solution of
$-\text{\rm div} \big( A_{n-2}^\e (x) \nabla u_{\e, n-2} \big)=F $, where
$A^\e_{n-2} (x) =A_{n-2} (x, x/\e_1, \dots, x/\e_{n-2} )$ and 
$A_{n-2}(x, y_1, \dots, y_{n-2})$ is the effective tensor  for 
\newline $A_{n-1} (x, y_1, \dots, y_{n-2}, y_{n-1})$,
with $(x, y_1, \dots, y_{n-2})$ fixed.
Continue the process until one reaches the tensor $A_0(x)=\widehat{A} (x)$.
By an induction argument on $n$,
to carry out the process above, 
it suffices to consider the special case where $n=1$ and $A^\e(x)=A(x, x/\e)$.
Moreover, by using a convolution in the $x$ variable,
one may  assume that $A=A(x, y)$ is Lipschitz continuous in $x\in \mathbb{R}^d$.
We  point out  that even though the case $A^\e (x)=A(x/\e)$  has been well studied,
new techniques are needed for the case $A^\e(x)  =A(x, x/\e)$  to  derive estimates with   sharp bounding constants depending 
explicitly on $\|\nabla_x A\|_\infty$.
For otherwise, the results would not be useful in the induction argument.

In the second step, 
a  rescaling argument, together with  another  induction argument,  is used 
to reach the finest scale $\e_n$.
 We mention  that  the condition (\ref{w-s-cond}) is  only used in the first step.
 Without this condition, our argument  yields 
 estimates (\ref{relipth}) and (\ref{reblipth}) for
 \begin{equation}\label{l-range}
 \e_1 + \left( \e_2/\e_1 +\cdots + \e_n/\e_{n-1}\right)^N \le r<R\le 1,
 \end{equation}
 where $N\ge 1$, with bounding constants $C$ depending on $N$.
 See Remark \ref{remark-6.1}.
 
As a byproduct of the first step described above,
we show  that if $A^\e(x)=A(x, x/\e)$, then
\begin{equation} \label{rate-basic}
\| u_\e -u_0\|_{L^2(\Omega)}
\le C  \e\left\{ 1+ \|\nabla_x A\|_\infty +\e \|\nabla_x A\|_\infty^2  \right\}
\big( \| F\|_{L^2(\Omega)} + \| f\|_{H^{3/2}(\partial\Omega)} \big)
\end{equation}
for $0<\e<1$,
where $C$ depends only on $d$, $m$, $\mu$, and $\Omega$
(see Lemma \ref{LE61}).
Estimate (\ref{rate-basic}) improves a similar estimate in \cite{xndcds2019},
where a general  case $A^\e (x)=A(x, \rho(x) /\e)$ was considered by the first and third authors.
It  also leads to the following theorem on the $L^2$ convergence rate 
for the operator  $\mathcal{L}_\e$.

\begin{theorem} \label{tco}
 Let $\Omega$ be a bounded $C^{1,1}$ domain in $\mathbb{R}^d$. 
 Assume that  
 $A$ satisfies  \eqref{elcon}, \eqref{pcon}, and \eqref{lipcon} with $\theta=1$.  
 Let $\mathcal{L}_\e$ be given by (\ref{operator})
 with $0<\e_n < \e_{n-1}< \dots< \e_1<1$.
  For $F\in L^2(\Omega; \mathbb{R}^m)$ and
  $f\in H^{3/2}(\partial\Omega; \mathbb{R}^m)$,
  let $u_{\varepsilon} \in H^{1}(\Omega;\mathbb{R}^{m})$
    be the solution of \eqref{DP} 
    and $u_0$ the solution of the homogenized problem (\ref{DP-0}).
    Then
 \begin{align}\label{tre1}
 \|u_{\varepsilon}-u_{0}\|_{L^2(\Omega)}
 \leq C \big\{ \e_1 
 +\e_2/\e_1
 +\cdots +
 \e_{n}/\e_{n-1}
 \big\} \| u_0\|_{H^2(\Omega)}, 
 \end{align}
 where $C$ depends at most on $d$, $m$, $n$, $\mu$, $L$, 
 and $\Omega$.
\end{theorem}

In the case $A^\e=A(x/ \e, x/\e^2)$,
the estimate (\ref{tre1}) was proved in \cite{past2013} (also see \cite{past07, past2016}).
As indicated in \cite{past2016},
one  may  extend the  proof to the general case considered in Theorem \ref{tco}.
However, the error estimates of the multiscale expansions  for  the case $n=2$ in \cite{past2013}
are already quite involved,  and their extension to the case $n>2$ is not so obvious.
Our proof of (\ref{tre1}), which is based on the idea of reiterated homogenization,
seems to be natural and is much simpler conceptually.

The paper is organized as follows.
In Section 2 we  give the definition of the effective  tensor  $\widehat{A}(x) $ as well as the tensors $A_k (x, y_1, \dots, y_k)$ for $ 1\le k\le n$, mentioned earlier.
We also introduce a smoothing operator and prove two estimates needed in the following sections.
The proof of (\ref{rate-basic}) is given in Section 3 and that of Theorem \ref{tco} in Section 4.
In Section 5 we establish an approximation theorem, using the results in Section 3.
Sections 6 and 7  are devoted to the proofs  of Theorems \ref{lipth} and \ref{blipth}, respectively.

For notational simplicity we will assume $m=1$ in the rest of the paper.
However, no particular fact pertain to the scalar case is ever used.
All results and proofs extend readily to the case $m>1$ - the case of elliptic systems.
We will use $\fint_E u$ to denote the $L^1$ average of $u$ over the set $E$;
i.e. $\fint_E u=\frac{1}{|E|} \int_E u$.
A function is said to be 1-periodic in $y_k\in \mathbb{R}^d$  if it is periodic in $y_k$ with respect to $\mathbb{Z}^d$.
Finally, the summation convention is used throughout. 

%%%%%%%%%%%%%%%%%%%%%%%%%%%%%%%%
 
%%%%%%%%%%%%%%%%%%%%%%%%%%%%%%%%

\section{Preliminaries}\label{section-2}

\subsection{Effective coefficients} \label{section-2.1}

Suppose $A=A(x, y_1, \dots, y_n)$ satisfies conditions (\ref{elcon}) and (\ref{pcon}).
To define the effective matrix $\widehat{A} =\widehat{A} (x)$ in the homogenized operator 
$
\mathcal{L}_0 =-\text{\rm div} \big(\widehat{A} (x)\nabla \big),
$
we introduce a sequence of  $d\times d$ matrices,
\begin{equation} \label{A-n}
A_\ell =A_\ell (x, y_1, \dots, y_\ell) \quad \text{ for } 0\le \ell \le n,
\end{equation}
which are  1-periodic in $(y_1, \dots, y_\ell)\in \mathbb{R}^{d\times \ell}$ and satisfy the ellipticity condition,
\begin{equation}\label{ellipticity-1}
\|A_\ell \|_{L^\infty(\mathbb{R}^{d\times (\ell +1)})}
\le \mu_1 \quad \text{ and } \quad
\mu |\xi|^2\le \langle A_\ell \xi, \xi \rangle
\end{equation}
for $\xi\in \mathbb{R}^d$, where $\mu_1>0$ depends only on $d$, $n$ and $\mu$.
To this end, we let
$
A_n (x, y_1, \cdots, y_n)=A(x, y_1, \dots, y_n).
$
Suppose $A_\ell$ has been given for some $1 \le \ell \le n$. 
For a.e. $(x, y_1, \dots, y_{\ell-1})\in \mathbb{R}^{d\times \ell }$ fixed,
 we  solve the elliptic cell problem,
\begin{equation}\label{cell-1}
\left\{
\aligned
& -\text{\rm div}_y  \big( A_\ell  (x, y_1, \dots, y_{\ell-1}, y ) \nabla_y  \chi_\ell ^j)
=\text{\rm div}_y \big( A_\ell  (x, y_1, \dots, y_{\ell-1}, y ) \nabla_y  y ^j \big) \quad \text{ in } \mathbb{T}^d,\\
& \chi_\ell ^j=\chi_\ell ^j (x, y_1, \cdots, y_{\ell-1}, y ) \text{ is 1-periodic in } y,\\
 & \int_{\mathbb{T}^d} \chi_\ell ^j  (x, y_1, \dots, y_{\ell-1}, y )\, dy  =0
\endaligned
\right.
\end{equation}
for $1\le j\le d$,
where $y^j$ denotes the $j$th component of $y \in \mathbb{R}^d$.
Since $A_\ell$ is 1-periodic in $(y_1, \dots, y_\ell)$,
so is the corrector $\chi_\ell (x, y_1, \dots, y_{\ell-1}, y_\ell)=(\chi_\ell^1, \cdots, \chi_\ell^d)$.
We now define
\begin{equation}\label{A-ell}
A_{\ell-1} (x, y_1, \dots, y_{\ell-1})
=\fint_{\mathbb{T}^d}
\Big( 
A_\ell (x, y_1, \dots, y_\ell)
+A_\ell (x, y_1, \dots, y_\ell) \nabla_{y_\ell} \chi_\ell \Big) dy_\ell.
\end{equation}
Clearly, $A_{\ell-1}$ is 1-periodic in $(y_1, \dots, y_{\ell-1})$.
It is also well known that $A_{\ell-1}$ satisfies the ellipticity condition (\ref{ellipticity-1}) \cite{lions1978}.
As a result, by induction, we obtain the matrix $A_\ell$ for $0\le \ell \le n$. In particular, 
 $\widehat{A} (x) =A_0(x)$ is the effective matrix for the operator $\mathcal{L}_\e$ in (\ref{operator}).

 \begin{theorem}\label{lem2.2}
Suppose $A$ satisfies  conditions (\ref{elcon}) and (\ref{pcon}).
Also assume that as a function of $(x, y_1, \dots, y_{n-1})$,
$ A\in  C(\mathbb{R}^{d \times n}; L^\infty(\mathbb{R}^d)).
$
Let $\Omega$ be a bounded Lipschitz domain in $\mathbb{R}^d$.
 Let $u_\e$ be a weak solution of the Dirichlet problem (\ref{DP}),
 with $F\in H^{-1}(\Omega)$ and $f\in H^{1/2}(\partial\Omega)$.
 Then,
 if $\e\to 0$ and
 $(\e_1, \e_2, \dots, \e_n)$ satisfies the condition (\ref{s-condition}),
  $u_\varepsilon$ converges weakly in $H^1(\Omega)$ to the solution $u_0$ of  
 the homogenized problem \eqref{DP-0}. 
\end{theorem}

Theorem \ref{lem2.2},
whose proof may be found in  \cite{lions1978,Allaire1996Multiscale},
is not used in this paper.
In fact, by approximating the coefficients,
our quantitative result in Theorem \ref{tco}, provides another proof of Theorem \ref{lem2.2}.

It follows by the energy estimate as well as Poincar\'e's  inequality that
\begin{equation}\label{3.000}
\fint_{\mathbb{T}^d}
|\nabla_y \chi_\ell (x, y_1, \dots, y_{\ell-1}, y_\ell )|^2\, dy_\ell
+\fint_{\mathbb{T}^d}
|\chi_\ell  (x, y_1, \dots, y_{\ell-1},  y_\ell )|^2\, dy_\ell
\le C
\end{equation}
for a.e. $(x, y_1, \dots, y_{\ell-1}) \in \mathbb{R}^{d\times \ell }$,
where $1 \le \ell \le n$ and $C$ depends only on $d$, $n$ and $\mu$.
The next lemma gives the H\"older estimates for $\chi_\ell$ and $A_\ell$ under the H\"older continuity condition on $A$.

\begin{lemma}\label{le2.0}
 Suppose $A$ satisfies conditions (\ref{elcon}), (\ref{pcon}),  and (\ref{lipcon})
 for some $\theta \in (0, 1]$ and $L\ge 0$.
 Then  
 \begin{equation}\label{est-corr}
 \aligned
 & \| \chi_\ell (x, y_1, \dots, y_{\ell-1}, \cdot) 
 -\chi_\ell (x^\prime, y_1^\prime, \dots, y_{\ell-1}^\prime, \cdot )\|_{H^1(\mathbb{T}^d)}\\
 &\qquad\qquad
\le  CL \big( |x-x^\prime| +|y_1-y_1^\prime|
 +\cdots + |y_{\ell-1}  -y_{\ell-1}^\prime|   \big)^\theta,\\
&  | A_{\ell-1}  (x, y_1, \dots, y_{\ell-1})
 -A_{\ell-1} (x^\prime, y_1^\prime, \dots, y^\prime_{\ell-1})  |\\
 & 
 \qquad\qquad
 \le  
 CL \big( |x-x^\prime| +|y_1-y_1^\prime|
 +\cdots + |y_{\ell-1} -y_{\ell-1}^\prime|  \big)^\theta
 \endaligned
 \end{equation}
 for $1\le \ell \le n$,
  where $C$ depends only on $d$, $n$, $\theta$ and $\mu$.
 \end{lemma}
 
\begin{proof}
 It suffices  to prove (\ref{est-corr}) for $\ell=n$.
 The rest follows by induction.
Note that for $(x, y_1, \dots, y_{n-1}), (x^\prime, y_1^\prime, \dots, y_{n-1}^\prime ) \in \mathbb{R}^{d\times n}$ fixed,
\begin{align*}
&-\text{\rm div}_{y}\Big(  A (x, y_1, \dots, y_{n-1}, y)
\nabla_{y} \big(
\chi_n^j (x,y_1,\dots  y_{n-1}, y)
- \chi_n^j (x^\prime, y_1^\prime, \dots, y_{n -1}^\prime, y )\big)\Big)\\ 
& = \text{div}_{y}
\Big( 
\big({A}(x,y_1, \dots, y_{n-1}, y)
-A(x^\prime, y_1^\prime, \dots, y_{n-1}^\prime, y)
\big) 
\nabla_y \big(y^j + \chi^j_n (x^\prime, y_1^\prime, \dots, y^\prime_{n-1},y)\big) \Big).
\end{align*}
The estimate for the correct $\chi_n$ in (\ref{est-corr}) follows readily from the usual energy estimate
and (\ref{lipcon}).
In view of (\ref{A-ell}) we may deduce
the estimate for $A_{n-1}$ in (\ref{est-corr})
by using (\ref{lipcon}) and the estimate of $\chi_{n}$ in (\ref{est-corr}).
\end{proof}

%%%%%%%%%%%%%%%%%%%%%%

\subsection{ An $\e$-smoothing operator}

 Fix a function $\varphi \in C_{0}^{\infty}(B(0,1/2))$ such that $\varphi\geq 0$ and $\int_{\mathbb{R}^{d}}\varphi  dx=1$.
  For functions  of form $g^\e (x)= g(x, x/\e)$,
  we introduce a smoothing operator $S_\e$,  defined by
\begin{align}\label{smoothing}
  S_\e (g^\e)(x)
  =\int_{\mathbb{R}^d} 
  g(z, x/\e)\varphi_\e (x-z) dz ,
\end{align}
 where $\varphi _{\e }(z)=\e^{- d}\varphi (z / \e)$.
 Note that the smoothing is only done to the slow variable $x$.
 
 \begin{lemma}\label{s-lemma-1}
 Let $1\le p<\infty$.
 Suppose that $h=h(x,y)$ is 1-periodic in $y$  and $h\in L^\infty(\mathbb{R}^d_x; L^p(\mathbb{T}^d_y))$.
 Then for any $f\in L^p(\mathbb{R}^d)$, 
 \begin{equation}\label{s-1}
 \| S_\e (h^\e f)\|_{L^p(\mathbb{R}^d)}
 \le C \| f\|_{L^p(\mathbb{R}^d)}
 \sup_{x\in \mathbb{R}^d}
 \left(\fint_{\mathbb{T}^d} |h(x, y)|^p\, dy\right)^{1/p},
 \end{equation}
 where  $h^\e(x) =h (x, x/\e)$ and $C$ depends only on $d$ and $p$.
 \end{lemma}
 
 \begin{proof}
 It follows by  H\"older's inequality and the assumption $\int_{\mathbb{R}^d} \varphi=1$  that
 $$
 |S_\e (h^\e f) (x)|^p
 \le \int_{\mathbb{R}^d}
 |h(z, x/\e)|^p | f(z)|^p \varphi_\e (x-z)\, dz.
 $$
 This, together with Fubini's Theorem, gives
 $$
 \aligned
 \int_{\mathbb{R}^d}
 | S_\e (h^\e f)|^p\, dx
 &\le \int_{\mathbb{R}^d}
  |f(z)|^p \int_{\mathbb{R}^d} 
 \varphi_\e (x-z) |h(z, x/\e)|^p\, dx \, dz\\
 & \le  \| f\|_{L^p(\mathbb{R}^d)}^p
 \sup_{z\in \mathbb{R}^d}
 \int_{\mathbb{R}^d}
 \varphi_\e (x-z) |h(z, x/\e)|^p\, dx
\\
 & \le C  \| f\|_{L^p(\mathbb{R}^d)}^p\sup_{z\in \mathbb{R}^d}
 \fint_{B(z, \e/2)} 
 |h(z, x/\e)|^p\, dx.
 \endaligned
 $$
 Using the periodicity of $h(x, y)$ in the second variable,
 it is easy to  see that
 $$
 \sup_{z\in \mathbb{R}^d}
 \fint_{B(z, \e/2)} 
 |h(z, x/\e)|^p\, dx
\le C 
 \sup_{x\in \mathbb{R}^d}
 \fint_{\mathbb{T}^d} |h(x, y)|^p\, dy,
 $$
 which  finishes the proof.
 \end{proof}
 
\begin{lemma}\label{s-lemma-2}
Let $1\le p\le \infty$.
Suppose that $h=h(x, y)\in L^\infty(\mathbb{R}^d \times \mathbb{R}^d)$
and $\nabla_x h\in L^\infty(\mathbb{R}^d \times \mathbb{R}^d)$.
Then for any $f\in W^{1, p} (\mathbb{R}^d)$,
\begin{equation}\label{s-2}
\| h^\e f -S_\e (h^\e f) \|_{L^p(\mathbb{R}^d)}
\le C \e
\Big\{ \|\nabla_x h\|_\infty \| f\|_{L^p(\mathbb{R}^d)}
+ \| h\|_\infty
\|\nabla f \|_{L^p(\mathbb{R}^d)}
\Big\},
\end{equation}
where $h^\e (x)=h(x, x/\e)$ and
$C$ depends only on $d$ and $p$.
\end{lemma}

\begin{proof}
Write
$$
h^\e (x) f(x)
-S_\e (h^\e f) (x)
=
\int_{\mathbb{R}^d}
\big( h(x, x/\e) f(x) - h(z, x/\e) f(z) \big)
\varphi_\e (x-z)\, dz,
$$
which leads to
$$
|h^\e (x) f(x)
-S_\e (h^\e f) (x)|
  \le C \fint_{B(x, \e/2)}
| h(x, x/\e) f(x) - h(z, x/\e) f(z)|\, dz.
$$
We now apply the inequality,
\begin{equation}\label{ineq}
\fint_{B(x,\e/2)}
|u(z)-u(x)|\, dz
\le C
\int_{B(x, \e/2 )}
\frac{|\nabla u(z)|}{|z-x|^{d-1}}\, dz,
\end{equation}
where $C$ depends only  on $d$.
This gives
$$
\aligned
& |h^\e (x) f(x)
-S_\e (h^\e f) (x)|\\
 & \le C  \|\nabla_x h\|_\infty
\int_{B(x, \e/2)}
\frac{|f(z)|}{|z-x|^{d-1}}\, dz
+ C  \| h\|_\infty
\int_{B(x, \e/2)} 
\frac{|\nabla_z f(z)|}{|z-x|^{d-1}}\, dz.
\endaligned
$$
It follows that
\begin{equation}\label{s-3}
\aligned
\int_{\mathbb{R}^d}
|h^\e f -S_\e (h^\e f) | |F|\, dx
&\le C \|\nabla_x h \|_\infty
\int_{\mathbb{R}^d} \left(\int_{B(x, \e/2)}
\frac{| f(z)||F(x)|}{|z-x|^{d-1}} dz \right) dx\\
&\qquad
+ C \| h \|_\infty
\int_{\mathbb{R}^d} \left(\int_{B(x, \e/2)}
\frac{| \nabla_z f(z)||F(x)|}{|z-x|^{d-1}} dz \right) dx.
\endaligned
\end{equation}
Finally, we note that the operator defined by
$$
Tg (x) =\int_{B(x, \e/2)}
\frac{g(z)}{|z-x|^{d-1}}\, dz
$$
is bounded on $L^p(\mathbb{R}^d)$ and $\|Tg\|_{L^p(\mathbb{R}^d)}
\le C\e \| g \|_{L^p(\mathbb{R}^d)}$ for $1\le p \le \infty$.
Thus, if $1\le p\le\infty$ and $q=p^\prime$,
$$
\int_{\mathbb{R}^d}
|h^\e f -S_\e (h^\e f) | |F|\, dx
\le C \e \| F\|_{L^q(\mathbb{R}^d)}
\Big\{
 \|\nabla_x h\|_\infty
\| f\|_{L^p(\mathbb{R}^d)}
+  \| h\|_\infty
\| \nabla f \|_{L^p(\mathbb{R}^d)}
\Big\},
$$
from which the inequality (\ref{s-2}) follows by duality.
\end{proof}

%%%%%%%%%%%%%%%%%%%%%%%%%%%%%%%%%%%%%%%%%%%%%

\section{Convergence rate $(n=1)$}\label{section-3}

In this section we consider a simple case, where $n=1$ and 
\begin{equation}\label{op-simple}
\mathcal{L}_\e
=-\text{\rm div} \big( A(x, x/\e)\nabla \big).
\end{equation}
The matrix  $A=A(x,y)$ satisfies the ellipticity condition (\ref{elcon}) and is 1-periodic in $y\in \mathbb{R}^d$.
We also assume that
\begin{equation}\label{lip-simple}
\|\nabla_x A\|_\infty
=\|\nabla_x A\|_{L^\infty(\mathbb{R}_x^d \times \mathbb{R}^d_y)} <\infty.
\end{equation}
Recall  that 
$$
\widehat{A}(x)=\fint_{\mathbb{T}^d}
\Big( A(x, y) + A(x, y) \nabla_y \chi (x, y) \Big) dy,
$$
where the corrector $\chi (x, y)= (\chi^1(x, y), \dots, \chi^d (x, y))$ is given by the cell problem 
(\ref{cell-1}) with $\ell =n=1$.
Note that by (\ref{est-corr}), 
\begin{equation}\label{3.0}
\|\nabla_x \widehat{A} \|_\infty
\le C \| \nabla_x A \|_\infty,
\end{equation}
and
\begin{equation}\label{3.00}
\fint_{\mathbb{T}^d}
\big( |\nabla_x \nabla_y \chi (x, y)|^2 
+|\nabla_x \chi(x, y)|^2\big) dy
\le C \|\nabla_x A\|^2_\infty,
\end{equation}
where $C$ depends only on $d$ and $\mu$.

Define
\begin{equation}\label{B}
B(x, y)= A(x, y)+A(x, y) \nabla_y \chi(x, y)-\widehat{A} (x).
\end{equation}
The $d\times d$ matrix $B(x, y)=(b_{ij} (x, y)) $ is 1-periodic in $y$ and
\begin{equation}\label{3.81}
\fint_{\mathbb{T}^d}
|B(x, y)|^2\, dy\le C,
\end{equation}
where $C$ depends only on $d$ and $\mu$.
In view of (\ref{3.0})-(\ref{3.00}) we obtain 
\begin{equation}\label{est-B}
\fint_{\mathbb{T}^d}
|\nabla_x B(x, y)|^2\, dy 
\le C \|\nabla_x A\|_\infty^2.
\end{equation}
By the definitions of $\widehat{A} (x) $ and $\chi(x, y)$,
it follows that
\begin{equation}\label{3.01}
\int_{\mathbb{T}^d} b_{ij} (x, y)\, dy=0
\quad
\text{ and } \quad
\frac{\partial}{\partial y^i} b_{ij} (x, y)=0
\end{equation}
for  each $x\in \mathbb{R}^d$ (the index $i$ is summed from $1$ to $d$), 
where we have used the notation $y=(y^1, \cdots, y^d)\in \mathbb{R}^d$.

\begin{lemma}\label{lemma-3.1}
There exist functions $\phi(x, y)=(\phi_{kij} (x, y)) $
with $1\le k, i, j\le d$ such that $\phi $ is 1-periodic in $y$,
\begin{equation}\label{3.02}
\phi_{kij}=-\phi_{ikj} \quad 
\text{ and } \quad
b_{ij} (x, y)
=\frac{\partial }{\partial y^k} \phi_{kij} (x, y).
\end{equation}
Moreover,   $\int_{\mathbb{T}^d} \phi  (x, y) dy=0$, and
\begin{equation}\label{3.03}
\aligned
\fint_{\mathbb{T}^d}
|\nabla_y \phi(x, y)|^2\, dy
+\fint_{\mathbb{T}^d} |\phi(x, y)|^2\, dy &  \le C,\\
\fint_{\mathbb{T}^d}
|\nabla_x \nabla_y \phi(x, y)|^2\, dy
+\fint_{\mathbb{T}^d}  |\nabla_x \phi(x, y) |^2 \, dy &  \le C \|\nabla_x A\|_\infty^2 ,
\endaligned
\end{equation}
where $C$ depends only on $d$ and $\mu$.
\end{lemma}

\begin{proof}
Using (\ref{3.01}),
 the flux correctors $\phi_{kij}$  are constructed  in the same manner
as in the case $A=A(y)$
(see e.g. \cite{shennote2017}).
Indeed, for each $x$ fixed, one solves the cell problem
$$
\left\{
\aligned
& \Delta_y f_{ij} (x, y)= b_{ij} (x, y) \quad \text{ in } \mathbb{T}^d,\\
& f_{ij} (x, y) \text{ is 1-periodic in } y,
\endaligned
\right.
$$
and sets
$$
\phi_{kij} (x, y)
=\frac{\partial}{\partial y^k}
f_{ij} (x, y) -\frac{\partial}{\partial y^i} f_{kj} (x, y).
$$
The first inequality in (\ref{3.03})  follows by using the $L^2$ estimate and (\ref{3.81}).
To see the second one uses (\ref{est-B}).
\end{proof}

Let $u_\e$ be a weak solution of the Dirichlet problem (\ref{DP}) and
$u_0$ the solution of the homogenized problem (\ref{DP-0}).
Let
\begin{equation}\label{w}
w_\e
=u_\e -u_0 
-\e S_\e (\eta_\e  \chi^\e \nabla u_0 ),
\end{equation}
where $\chi^\e (x)=\chi(x, x/\e)$ and the operator  $S_\e$ is defined by (\ref{smoothing}).
The cut-off function $\eta_\e $  in (\ref{w}) is chosen so that
$\eta_\e  \in C_0^\infty(\Omega)$, $0\le \eta_\e\le 1$,
$$
\aligned
& \eta_\e (x)=1\quad \text{  if  } x\in \Omega \text{  and dist} (x,\partial\Omega) \ge  4\e,\\
& \eta_\e  (x)=0 \quad \text{  if dist} (x, \partial\Omega)\le 3\e,
\endaligned
$$
and $|\nabla \eta_\e| \le C \e^{-1}$. 
Define 
\begin{equation}\label{O-t}
\Omega_t =\big\{ x\in \Omega: \ \text{ dist}(x, \partial\Omega)< t \big\}.
\end{equation}

The following lemma was proved in \cite{shenzhu2017} for the case $A^\e=A(x/\e)$.
The case $A^\e =A(x, \rho(x) /\e)$ for stratified structures  was considered in \cite{xndcds2019} by the first and third authors. Also see \cite{Xu-nonlinear} for the nonlinear case.
The estimate (\ref{bl-1}) is sharper than the similar estimates in \cite{xndcds2019, Xu-nonlinear}.

\begin{lemma}\label{lemma-3.3}
 Let $\Omega$ be a bounded Lipschitz domain in $\mathbb{R}^d$.
 Let $w_\e$ be defined by (\ref{w}).
   Then for any $\psi\in H^1_0(\Omega)$,
\begin{equation}\label{bl-1}
\aligned
    & \Big|\int_\Omega A^\e
    \nabla w_\varepsilon\cdot \nabla \psi dx\Big| \\
&    \leq 
C\e  \|\nabla\psi\|_{L^2(\Omega)}
\Big\{
\|\nabla_x A\|_\infty
\|\nabla u_0\|_{L^2(\Omega)}
+ \|\nabla^2 u_0\|_{L^2(\Omega\setminus\Omega_{3\e})} \Big\} \\
&\qquad\qquad
+    C 
    \|\nabla \psi \|_{L^2(\Omega_{5\e})}
    \|\nabla u_0\|_{L^2(\Omega_{4\e})},   
 \endaligned
 \end{equation}
where $A^\e =A(x, x/\e)$ and 
$C$ depends only on $d$, $\mu$, and $\Omega$.
\end{lemma}

\begin{proof}
Using $\mathcal{L}_\e (u_\e)=\mathcal{L}_0 (u_0)$, we obtain 
\begin{equation}\label{3.10}
\aligned
\mathcal{L}_\e (w_\e)
&= \text{\rm div} \big[ ( A^\e-\widehat{A}) \nabla u_0 \big]
+\text{\rm div} \big[
A^\e  S_\e  \big( \eta_\e (\nabla_y \chi)^\e \nabla u_0 \big) \big]\\
& 
\quad + \e\,  \text{\rm div}
\big[ A^\e S_\e \big( (\nabla \eta_\e ) \chi^\e \nabla u_0\big) \big]
+\e\,  \text{\rm div} \big[
A^\e S_\e \big(   \eta_\e (\nabla_x \chi)^\e \nabla u_0\big) \big]\\
&\quad +
\e\,  \text{\rm div} 
\big[ A^\e  S_\e \big( \eta_\e \chi^\e \nabla^2 u_0\big) \big].
\endaligned
\end{equation}
The last three terms  in the right-hand side of (\ref{3.10}) are easy to handle.
Let $B(x, y)$ be given  by (\ref{B}).
To deal with the first two terms,
we write the sum of them as
\begin{equation}\label{3.11}
I_1 +I_2 
+\text{\rm div} \big[ S_\e \big(  \eta_\e B^\e \nabla u_0) \big],
\end{equation}
where $B^\e =B(x, x/\e)$, and 
\begin{equation}\label{3.12}
\aligned
I_1 &=
\text{\rm div}
\big[ (A^\e -\widehat{A} ) \nabla u_0
-S_\e \big( (A^\e -\widehat{A} ) \eta_\e \nabla u_0 \big) \big],\\
I_2 &=
\text{\rm div}
\big[ A^\e S_\e \big( \eta_\e (\nabla_y \chi)^\e \nabla u_0 \big)
-S_\e \big( \eta_\e A^\e (\nabla_y \chi)^\e \nabla u_0 \big) \big].
\endaligned
\end{equation}
It follows from (\ref{3.10})-\ref{3.12}) that
\begin{equation}\label{3.13}
\aligned
 &  \Big|\int_\Omega A^\e
    \nabla w_\varepsilon\cdot \nabla \psi dx\Big| \\
    &
    \le \int_\Omega
    \big| (A^\e -\widehat{A} ) \nabla u_0
-S_\e \big( (A^\e -\widehat{A} ) \eta_\e \nabla u_0 \big) \big| |\nabla \psi|\, dx\\
&\qquad + \int_\Omega
\big| A^\e S_\e \big( \eta_\e (\nabla_y \chi)^\e \nabla u_0 \big)
-S_\e \big( \eta_\e A^\e (\nabla_y \chi)^\e \nabla u_0 \big) \big| |\nabla \psi|\, dx\\
&\qquad  +\Big|
\int_\Omega
S_\e \big( \eta_\e  B^\e \nabla u_0) \cdot \nabla \psi \, dx \Big|\\
& \qquad + C \e \int_\Omega 
|S_\e \big( (\nabla \eta_\e ) \chi^\e \nabla u_0\big) | |\nabla \psi|\, dx\\
& \qquad + C \e \int_\Omega
|S_\e \big( \eta_\e (\nabla_x \chi)^\e \nabla u_0\big)| |\nabla \psi|\, dx\\
& \qquad + C \e \int_\Omega
|S_\e \big( \eta_\e \chi^\e \nabla^2 u_0\big)| |\nabla \psi|\, dx\\
&=J_1+\dots + J_6,
\endaligned
\end{equation}
for any $\psi\in H_0^1 (\Omega)$.
We estimate $J_i, i=1, \dots, 6$ separately.

To bound $J_4$, we use the Cauchy inequality and (\ref{s-1}) to obtain 
\begin{equation}\label{3.14}
\aligned
J_4
&\le C \e \| S_\e \big( (\nabla \eta_\e ) \chi^\e \nabla u_0\big) \|_{L^2(\Omega)}
\|\nabla \psi\|_{L^2(\Omega_{5\e} )} \\
&\le C \e \| (\nabla \eta_\e ) \nabla u_0\|_{L^2(\Omega)}
\|\nabla \psi \|_{L^2(\Omega_{5\e} )}\\
&\le C 
\|\nabla u_0\|_{L^2(\Omega_{4\e})}
\|\nabla \psi\|_{L^2(\Omega_{5\e} )},
\endaligned
\end{equation}
where we have used  the estimate for $\chi (x, y)$ in (\ref{3.000}).
In view of the estimate for $\nabla_x \chi(x, y)$ in (\ref{3.00}),
the same argument also shows that 
\begin{equation}\label{3.15}
J_5 +J_6
\le C \e \|\nabla \psi\|_{L^2(\Omega)} \big\{
  \|\nabla_x A\|_\infty
\|\nabla u_0\|_{L^2(\Omega)}
+  \|\nabla^2 u_0\|_{L^2(\Omega\setminus \Omega_{3\e})}\big\}.
\end{equation}

Next, to bound $J_3$, we use the flux correctors $\phi_{kij}$ given by Lemma \ref{lemma-3.1}.
Note that by using the second equation in  (\ref{3.02}),
$$
\aligned
&\eta_\e (x-z) b_{ij}(x-z, x/\e) \frac{\partial u_0}{\partial x^j} (x-z)\\
&
=\e \eta_\e (x-z) \frac{\partial}{\partial x^k}
\Big\{ \phi_{kij} (x-z, x/\e) \Big\}
\frac{\partial u_0}{\partial x^j} (x-z)\\
& \qquad 
-\e \eta_\e (x-z) \frac{\partial \phi_{kij}}{\partial x^k} (x-z, x/\e) \frac{\partial u_0}{\partial x^j} (x-z)\\
&= \e \frac{\partial}{\partial x^k}
\Big\{
\eta_\e (x-z) \phi_{kij} (x-z, x/\e) \frac{\partial u_0}{\partial x^j} (x-z) \Big\}\\
& \qquad
-\e \frac{\partial}{\partial x^k}
\Big\{ \eta_\e (x-z) \Big\}
\phi_{kij} (x-z, x/\e) \frac{\partial u_0}{\partial x^j} (x-z) \\
& \qquad
-\e \eta_\e (x-z) \frac{\partial \phi_{kij}}{\partial x^k} (x-z, x/\e) \frac{\partial u_0}{\partial x^j} (x-z)\\
& 
\qquad
-\e \eta_\e (x-z)
\phi_{kij}
(x-z, x/\e)
\frac{\partial^2 u_0}{\partial x^j\partial x^k} (x-z).
\endaligned
$$
It follows that
\begin{equation}\label{3.16}
\aligned
J_3 & =\e  \Big|  \int_\Omega \frac{\partial}{\partial x^k}
S_\e  \left(\eta_\e \phi_{kij}^\e \frac{\partial u_0}{\partial x^j} \right) \frac{\partial \psi}{\partial x_i}\, dx
 -\int_\Omega S_\e ( (\nabla \eta_\e)  \phi^\e \nabla u_0) \cdot \nabla \psi\, dx\\
& \qquad
- \int_\Omega S_\e (\eta_\e  (\nabla_x \phi)^\e \nabla u_0) \cdot \nabla \psi\, dx
- \int_\Omega S_\e (\eta_\e \phi^\e \nabla^2 u_0) \cdot \nabla \psi\, dx \Big|.
\endaligned
\end{equation}
By using the skew-symmetry property of $\phi_{kij}$ in (\ref{3.02}) and integration by parts we may show that
the first term  in the right-hand side of (\ref{3.16}) is zero,  if  $\psi\in C_0^\infty(\Omega)$.
The same is true for any $\psi\in H_0^1(\Omega)$ by a simple density argument.
The remaining terms in the right-hand side of (\ref{3.16}) may be handled as in the case of $J_4$,
but using estimates of $\phi$ and $\nabla_x \phi$ in (\ref{3.03}).
As a result, we obtain
\begin{equation}\label{3.17}
\aligned
J_3
 & \le C 
\|\nabla \psi\|_{L^2(\Omega_{5\e})}
 \|\nabla u_0\|_{L^2(\Omega_{4\e})}\\
&\qquad
+ C \e \|\nabla \psi\|_{L^2(\Omega)} \Big\{
 \|\nabla_x A\|_\infty
\|\nabla u_0\|_{L^2(\Omega)}
+ \|\nabla^2 u_0\|_{L^2(\Omega \setminus \Omega_{3\e})}
\Big\}.
\endaligned
\end{equation}

It remains to  estimate $J_1$ and $J_2$.
Note that
\begin{equation}\label{3.18}
\aligned
J_1
 & \le  C \int_\Omega |\nabla u_0| |1-\eta_\e | |\nabla \psi|\, dx
+ \int_\Omega
|(A^\e -\widehat{A}) \eta_\e \nabla u_0
-S_\e \big( (A^\e-\widehat{A} ) \eta_\e 
\nabla u_0 \big)| \, |\nabla \psi |\, dx\\
&=J_{11} +J_{12}.
\endaligned
\end{equation}
By the Cauchy inequality,
\begin{equation}\label{3.19}
J_{11}
\le C \|\nabla \psi\|_{L^2(\Omega_{4\e})}
\|\nabla u_0\|_{L^2(\Omega_{4\e})}.
\end{equation}
To bound $J_{12}$, we use (\ref{s-3}) to obtain
$$
\aligned
J_{12}
 & \le C   \|\nabla_x A\|_\infty
\int_\Omega  |\nabla \psi (x)|
\int_{B(x,\e)}
\frac{\eta_\e (z) |\nabla u_0 (z)|}{|z-x|^{d-1}}\, dzdx\\
& \qquad
+ C   \int_\Omega |\nabla \psi(x)|
\int_{B(x, \e)}
\frac{|\nabla \eta_\e| |\nabla u_0(z)|
+ \eta_\e (z)| |\nabla^2 u_0 (z)|}{|z-x|^{d-1}}\, dz dx.
\endaligned
$$
As in the proof of Lemma \ref{s-lemma-2}, this yields that
\begin{equation}\label{3.20}
\aligned
J_{12}
& \le C \e \|\nabla_x A\|_\infty
\|\nabla \psi\|_{L^2(\Omega)} \|\nabla u_0\|_{L^2(\Omega)}
+ C \|\nabla \psi\|_{L^2(\Omega_{5\e})}
\|\nabla u_0\|_{L^2(\Omega_{4\e})}\\
&\qquad\qquad
+ C \e
\|\nabla\psi\|_{L^2(\Omega)}
\|\nabla^2 u_0\|_{L^2(\Omega\setminus \Omega_{3\e})}.
\endaligned
\end{equation}

Finally, to bound $J_2$,
we observe that
$$
\aligned
J_2
 & \le  C\int_\Omega
\fint_{B(x, \e)}
| A(x, x/\e) -A(z, x/\e)|
\eta_\e (z)
|\nabla_y \chi (z, x/\e)|
|\nabla u_0 (z)| |\nabla \psi ( x)|\, dz dx\\
&\le C \e \|\nabla_x A\|_\infty
\int_\Omega
\fint_{B(x, \e )}
\eta_\e  (z)
|\nabla_y \chi (z, x/\e)|
|\nabla u_0 (z)| |\nabla \psi ( x)|\, dz dx\\
& \le C\e \|\nabla_x A\|_\infty
\|\nabla \psi\|_{L^2(\Omega)}
\|\fint_{B(x, \e )}
|\nabla_y \chi(z, x/\e)| \eta_\e (z)  |\nabla u_0 (z)|\, dz \|_{L^2(\Omega)}\\
&\le 
C\e \|\nabla_x A\|_\infty
\|\nabla \psi\|_{L^2(\Omega)}
\Big\| \left( \fint_{B(x, \e )}
|\nabla_y \chi(z, x/\e)|^2  \eta_\e (z)  |\nabla u_0 (z)|^2 \, dz \right)^{1/2} \Big\|_{L^2(\Omega)},
\endaligned
$$
where we have used the Cauchy inequality for the last two inequalities.
By using Fubini's Theorem and (\ref{3.000}) we see that
$$
\Big\| \left( \fint_{B(x, \e)}
|\nabla_y \chi(z, x/\e)|^2  \eta_\e (z)  |\nabla u_0 (z)|^2 \, dz \right)^{1/2} \Big\|_{L^2(\Omega)}
\le C \|\nabla u_0\|_{L^2(\Omega)}.
$$
This gives
$$
J_2\le C\e \|\nabla_x A\|_\infty
\|\nabla \psi\|_{L^2(\Omega)}
\|\nabla u_0\|_{L^2(\Omega)},
$$
and completes the proof.
\end{proof}

The next theorem provides an error estimate in $H^1(\Omega)$.

\begin{theorem}\label{thm-3}
Let $\Omega$ be a bounded Lipschitz domain in $\mathbb{R}^d$.
Assume that $A$ satisfies the same conditions as in Lemma \ref{lemma-3.3}.
Let $w_\e$ be defined by (\ref{w}).
Then
\begin{equation}\label{H-1-est}
\|  w_\e\|_{H^1(\Omega)}
\le C   \e^{1/2} \| u_0\|^{1/2}_{H^2(\Omega)} \| \nabla u_0\|^{1/2}_{L^2(\Omega)}
+C\e \| u_0\|_{H^2(\Omega)}
+ C \e \|\nabla_x A\|_\infty \|\nabla u_0\|_{L^2(\Omega)}
\end{equation}
for $0<\e <1$,
where $C$ depends only on $d$, $\mu$ and $\Omega$.
\end{theorem}

\begin{proof}
Note that $w_\e\in H_0^1(\Omega)$ and $\| w_\e\|_{H^1(\Omega)} \le C \|\nabla w_\e\|_{L^2(\Omega)}$.
By taking $\psi =w_\e$ in (\ref{bl-1}) and using the ellipticity 
condition of $A$, 
we obtain
\begin{equation}\label{L-2-bl}
\|  w_\e\|_{H^1(\Omega)}
\le C \e \big\{
\|\nabla_x A\|_\infty  \|\nabla u_0\|_{L^2(\Omega)}
+ \|\nabla^2 u_0\|_{L^2(\Omega\setminus \Omega_{3\e} )} \big\}
+ C \| \nabla u_0\|_{L^2(\Omega_{4\e})}.
\end{equation}
This, together with the inequality 
\begin{equation}\label{bl-est}
\| v\|_{L^2(\Omega_t)}
\le C t^{1/2} \| v\|_{L^2(\Omega)}^{1/2}  \| v\|^{1/2}_{H^1(\Omega)}
\end{equation}
for $t>0$ and $v\in H^1(\Omega)$,
where $\Omega_t$ is defined by (\ref{O-t}),
gives (\ref{H-1-est}).
\end{proof}

\begin{remark}
{\rm
Let $\Omega$ be a bounded Lipschitz domain.
Let $u_\e$, $u_0$ and $w_\e$ be the same as in Theorem \ref{thm-3}.
Observe that
$$
\aligned
\| u_\e -u_0\|_{L^2(\Omega)}
&\le \| w_\e\|_{L^2(\Omega)}
+ \e \| S_\e \big( \eta_\e \chi^\e \nabla u_0\big) \|_{L^2(\Omega)}\\
&\le C \| w_\e\|_{H^1(\Omega)}
+ C \e \| \nabla u_0\|_{L^2(\Omega)},
\endaligned
$$
where we have used  (\ref{s-1}).
This, together with (\ref{L-2-bl}), yields
\begin{equation}\label{L-2-estimate}
\aligned
\| u_\e -u_0\|_{L^2(\Omega)}
 & \le C \e (\|\nabla_x A\|_\infty +1) \|\nabla u_0\|_{L^2(\Omega)}
+ C \e \|\nabla^2 u_0\|_{L^2(\Omega\setminus \Omega_{3\e})}\\
& \qquad\qquad
+ C \|\nabla u_0\|_{L^2(\Omega_{4\e})},
\endaligned
\end{equation}
where $C$ depends only on $d$, $\mu$ and $\Omega$.
Estimate (\ref{L-2-estimate}) is not sharp, but will be useful in the proof of Theorems \ref{lipth} and \ref{blipth}.
}
\end{remark}

\begin{remark}
{\rm 
Let $\Omega$ be a bounded $C^{1, 1}$ domain in $\mathbb{R}^d$.
Let $w_\e$ be defined by (\ref{w}),
where $u_\e$ and $u_0$ have the same data $F$ and $f$.
Then
\begin{equation}\label{r-3.1}
 \|  w_\e\|_{H^1(\Omega)}\\
 \le
C \e^{1/2}  \Big\{ 
 1  +  \|\nabla_x A\|_\infty^{1/2}
 +\e^{1/2}   \|\nabla_x A\|_\infty \Big\}
\left( 
\| F\|_{L^2(\Omega)}
+ \| f\|_{H^{3/2}(\partial\Omega)}
\right),
\end{equation}
where $C$ depends only on $d$, $\mu$ and $\Omega$.
This follows from (\ref{H-1-est}),
the energy estimate
$$
\| u_0\|_{H^1(\Omega)}
\le C 
\left( 
\| F\|_{L^2  (\Omega)}
+ \| f\|_{H^{1/2}(\partial\Omega)}
\right),
$$
and
the $H^2$ estimate for $\mathcal{L}_0$,
\begin{equation}
\| u_0\|_{H^2(\Omega)}
\le C  ( \|\nabla_x A\|_\infty +1)
\left( 
\| F\|_{L^2(\Omega)}
+ \| f\|_{H^{3/2}(\partial\Omega)}
\right).
\end{equation}
where $C$ depends only on $d$, $\mu$ and $\Omega$.
}
\end{remark}

%%%%%%%%%%%%%%%%%%%%%

%%%%%%%%%%%%%%%%%%%%%%%%%%

\section{Proof of Theorem \ref{tco}}\label{section-4}

The  proof of Theorem \ref{tco} is based on an approach of homogenization with a parameter.
We start with the case $n=1$ and $A^\e =A(x, x/\e)$, considered in the last section.

\begin{lemma}\label{LE61}
 Let $\Omega$ be a bounded $C^{1,1}$ domain in $\mathbb{R}^d$.
  Assume that $A=A(x,y)$ is 1-periodic in $y$
  and satisfies conditions  \eqref{elcon} and \eqref{lip-simple}.
 Let $u_\e$ be a weak solution of (\ref{DP}), with
 $\mathcal{L}_\e =-\text{\rm div} \big(A(x, x/\e)\nabla \big)$,
 and $u_0$ the solution of (\ref{DP-0}) with the same data $F\in L^2(\Omega)$ and $f\in H^{3/2}(\partial\Omega)$.
  Then
  \begin{equation}\label{4.1}
  \|u_\e-u_0\|_{L^{2}(\Omega)}\\
 \leq C  \e  
  \Big\{ 
  1 + 
  \|\nabla_x A\|_\infty
 +\e \|\nabla_x A\|^2 _\infty  \Big\}
 \left( \| F\|_{L^2(\Omega)}
 + \|f\|_{H^{3/2}(\partial\Omega)}\right)
  \end{equation}
  for $0<\e<1$,
where $C$ depends only on $d$, $n$, $\mu$ and $\Omega$.
\end{lemma}

\begin{proof}
Let $w_\e$ be given by (\ref{w}).
It follows from (\ref{s-1}) that
$$
\|S_\e ( \eta_\e \chi^\e \nabla u_0)\|_{L^{2}(\Omega)}\leq C\|\nabla u_0\|_{L^2(\Omega)}.
$$
Thus it suffices to  show that $\|w_\va\|_{L^2(\Omega)}$ is bounded by the right-hand side of (\ref{4.1}).
This is done by using (\ref{bl-1}) and a duality argument,
 as in  \cite{suslinaD2013}.
Let $ A^*(x,y)$ denote the adjoint of ${A}(x,y)$. 
Note that $A^*(x,y)$ satisfies the same conditions as ${A}(x,y)$.
 We denote the corresponding correctors and flux correctors  by $ \chi^*(x,y)$ and $\psi^*(x,y)$, respectively.
  Its matrix of effective coefficients is  given by 
  $ \widehat{{A}^*}$ $= (\widehat{{A}})^*$, the adjoint of $\widehat{A}$.
  
 For $G\in C_c^\infty(\Omega)$, let $v_\e$ be the weak solution of the following Dirichlet problem,
\begin{equation}\label{PLE612}
\begin{cases}
 -  \text{div}\left( {A}^*(x, x/\e) \na v_{\va} (x)  \right)=G  \quad\text{in } \Omega,   \\
 v_{\va}=0  \quad \text{on } \partial\Omega,
 \end{cases}
\end{equation}
and $v_0$ the homogenized solution.
Define
 \begin{align*}
\widetilde{ w}_{\va}(x)=&v_\va-v_0-\va S_\e \big( \widetilde{\eta}_\e (\chi^*)^\e  \nabla v_0\big),
\end{align*}
where $(\chi^*)^\e =\chi^* (x, x/\e)$ and 
$\widetilde{\eta}_\e \in C_0^\infty(\Omega)$ is a cut-off function 
such that $0\le \widetilde{\eta}_\e \le 1$,
$$
\widetilde{\eta}_\e (x)=1   \text{ in }  \Omega\setminus\Omega_{10\e},\quad
\widetilde{\eta}_\e (x)=0   \text{ in } \Omega_{ 8 \e },
$$
and $|\nabla \widetilde{\eta}_\e | \le C \e^{-1}$.
Note that
\begin{align}
  \Big|\int_\Omega w_\varepsilon\cdot G\,  dx\Big|
  &=\Big|\int_\Omega {A}^\e (x)\nabla w_\varepsilon\cdot\nabla v_\varepsilon \, dx\Big|\nonumber\\ 
  &\leq \Big|\int_\Omega {A}^\varepsilon(x)\nabla w_\varepsilon\cdot\nabla \widetilde{w}_\varepsilon \, dx\Big|
  +\Big|\int_\Omega {A}^\varepsilon(x)\nabla w_\varepsilon\cdot\nabla v_0 \, dx\Big|\nonumber\\ 
  &\qquad
  +\e \Big|\int_\Omega {A}^\varepsilon(x)\nabla w_\varepsilon\cdot
  \nabla\big[ S_\e \big( \widetilde{\eta}_\e
  (\chi^*)^\e \nabla v_0\big) \big] dx\Big|\nonumber\\
  &\doteq J_1+J_2+J_3.\label{PLE613}
\end{align}
We estimate $J_1$, $J_2$, and $J_3$ separately.

By using the Cauchy inequality and (\ref{r-3.1}), we obtain 
\begin{equation}\label{4.11}
\aligned
J_1 
& \le C \|\nabla w_\e \|_{L^2(\Omega)}
\|\nabla \widetilde{w}_\e \|_{L^2(\Omega)}\\
& \le
C \e  \Big\{ 1  +    \|\nabla_x A \|_\infty
+\e  \|\nabla_x A\|^2 _\infty  \Big\}
\left( \| F\|_{L^2(\Omega)}
+ \| f\|_{H^{3/2}(\partial\Omega)}\right)
\| G \|_{L^2(\Omega)},
\endaligned
\end{equation}
where we have also used the estimate
\begin{align}\label{4.12}
 \| \widetilde{w}_\va\|_{H^1_0(\Omega)} \leq C
\e^{1/2}  \Big\{
 1 +  \|\nabla_x A\|_\infty^{1/2} + 
 \e^{1/2}  \|\nabla_x A\|_\infty\Big\}
  \| G\|_{L^2(\Omega)}.
  \end{align}
  The proof of (\ref{4.12}) is the same as that of (\ref{r-3.1}).
  
  Next, we use (\ref{bl-1}) to obtain 
  \begin{equation}\label{4.13}
 \aligned
 J_2
 &\le C \e \|\nabla v_0\|_{L^2(\Omega)}
 \Big\{ \|\nabla_x A\|_\infty \|\nabla u_0\|_{L^2(\Omega)}
 +\|\nabla^2 u_0\|_{L^2(\Omega)} \Big\}\\
 &\qquad\qquad
 +C \|\nabla v_0\|_{L^2(\Omega_{5\e})} 
 \|\nabla u_0\|_{L^2(\Omega_{4\e})}.
 \endaligned
 \end{equation}
 Note that by (\ref{bl-est}), 
 $$
  \|\nabla v_0\|_{L^2(\Omega_{5\e})} 
 \|\nabla u_0\|_{L^2(\Omega_{4\e})}
 \le C \e \| \nabla v_0\|_{L^2(\Omega)}^{1/2} \| v_0\|_{H^2(\Omega)}^{1/2}
 \| \nabla u_0\|_{L^2(\Omega)}^{1/2}
 \| u_0\|_{H^2(\Omega)}^{1/2}.
 $$
 This, together with (\ref{4.13}) and the energy estimates and $H^2$ estimates for $\mathcal{L}_0$ and $\mathcal{L}_0^*$,
 gives
 \begin{equation}\label{4.14}
 J_2 \le C \e (1+\|\nabla_x A\|_\infty)
 \left( \| F\|_{L^2(\Omega)}
+ \| f\|_{H^{3/2}(\partial\Omega)}\right)
\| G \|_{L^2(\Omega)}.
 \end{equation}
 
The estimate of $J_3$ is similar to that of $J_2$.
By (\ref{bl-1}) we see that
$$
J_3
\le C\e^2  \| \nabla\big[ S_\e \big( \widetilde{\eta}_\e
  (\chi^*)^\e \nabla v_0\big) \big]\|_{L^2(\Omega)}
  \Big\{
  \|\nabla_x A\|_\infty \|\nabla u_0\|_{L^2(\Omega)}
  +\|\nabla^2 u_0\|_{L^2(\Omega)} \Big\},
  $$
  where we have used the fact $\widetilde{\eta}_\e=0$ on $\Omega_{8\e}$.
  Note that by (\ref{s-1}),
  $$
  \aligned
   &  \| \nabla\big[ S_\e \big( \widetilde{\eta}_\e
  (\chi^*)^\e \nabla v_0\big) \big]\|_{L^2(\Omega)}\\
  &\le \| S_\e \big[ 
  (\nabla \widetilde{\eta}_\e ) (\chi^*)^\e \nabla v_0 \big]\|_{L^2(\Omega)}
  + \| S_\e \big[ \widetilde{\eta}_\e (\nabla_x \chi^* )^\e \nabla v_0 \big] \|_{L^2(\Omega)}\\
&\qquad
  + \e^{-1} \| S_\e \big[  \widetilde{\eta}_\e (\nabla_y \chi^*)^\e \nabla v_0\big]\|_{L^2(\Omega)}
  + \| S_\e \big[ \widetilde{\eta}_\e (\chi^*)^\e \nabla^2 v_0\big] \|_{L^2(\Omega)}\\ 
  &\le C \e^{-1} \|\nabla v_0\|_{L^2(\Omega)}
  + C \|\nabla^2 v_0\|_{L^2(\Omega)}.
  \endaligned
  $$
  It follows that
$$
\aligned
J_3  & \le C \e \Big\{ \|\nabla v_0\|_{L^2(\Omega)}
+ \e \|\nabla^2 v_0\|_{L^2(\Omega)} \Big\}
\Big\{ \|\nabla_x A\|_\infty \|\nabla u_0\|_{L^2(\Omega)}
+ \|\nabla^2 u_0\|_{L^2(\Omega)} \Big\}\\
&\le C \e (1+\|\nabla_x A\|_\infty) (1+\e \|\nabla_x A\|_\infty)
 \left( \| F\|_{L^2(\Omega)}
+ \| f\|_{H^{3/2}(\partial\Omega)}\right)
\| G \|_{L^2(\Omega)}.
\endaligned
$$
By combining  the estimates of $J_1, J_2$ and $J_3$ we obtain
$$
\aligned
&   \Big|\int_{\Omega}w_\varepsilon\cdot G  \, dx\Big|\\
&   \le C  \e  
  \Big\{ 
  1 + 
  \|\nabla_x A\|_\infty
 +\e \|\nabla_x A\|^2 _\infty  \Big\} \big(  \| F \|_{L^2(\Omega)} +  \|f\|_{H^{3/2}(\partial\Omega)} \big)  \|G\|_{L^2(\Omega)},
\endaligned
$$
from which the desired estimate for $w_\e$  follows by duality.
\end{proof}

We are now in a position to give the proof of Theorem \ref{tco}.

\begin{proof}[\bf Proof of Theorem \ref{tco}]
We prove the theorem by using an induction argument on $n$.
The case $n=1$ follows directly from Lemma \ref{LE61}.
Suppose that the  theorem is true  for some $n-1$.
To prove the  theorem for $n$,
let  $u_\e$ be a weak solution of the Dirichlet problem (\ref{DP}) and $u_0$ the solution of the homogenized problem
(\ref{DP-0}) with the same data $(F, f)$.
Let $v_\e$ be the weak solution  to
\begin{equation}\label{DP-I}
-\text{\rm div} 
\big( A_{n-1}  (x, x/\e_1, \dots, x/\e_{n-1}) \nabla v_\e \big) =F
\quad \text{ in } \Omega
\quad \text{ and } \quad 
v_\e =f \quad \text{ on } \partial\Omega,
\end{equation}
where $A_{n-1}$ is defined by (\ref{A-ell}) with $\ell=n$ and $A_n=A$.
Note that 
$$
\|\nabla_{x, y_1, \dots, y_{n-2}} A_{n-1} \|_\infty
\le C \|\nabla_{x, y_1, \dots, y_{n-1}} A\|_\infty \le CL.
$$
By the induction assumption,
\begin{equation}\label{ind}
\| v_\e -u_0\|_{L^2(\Omega)}
\le C \big\{ 
\e_1
+\e_2/\e_1
+\cdots \e_{n-1} /\e_{n-2}  \big\}
\big\{
\| F\|_{L^2(\Omega)} + \| f\|_{H^{3/2}(\partial\Omega)} \big\},
\end{equation}
where $C$ depends only on $d$, $n$, $\mu$, $L$ and $\Omega$.

To bound $\| u_\e -v_\e\|_{L^2(\Omega)}$,
we use Lemma \ref{LE61}.
For each $0<\e<1$ fixed, we let
$$
E(x, y)=A(x, x/\e_1, \dots, x/\e_{n-1} , y).
$$
Then
$$
A(x, x/\e_1, \dots, x/\e_n)
=E(x, x/\e_n).
$$
Note that
$$
 \|\nabla_x E\|_\infty
\le C L \e_{n-1} ^{-1},
$$
 where we have used  the assumption  that  $0<\e_n<\e_{n-1} < \cdots< \e_1<1$.
 By Lemma \ref{LE61}, we obtain 
 $$
 \aligned
 \| u_\e -v_\e\|
  & \le C \e_n
 \left\{ 1+ \|\nabla_x E\|_\infty +\e_n  \| \nabla_x E\|^2_\infty \right\}
 \left\{ \| F\|_{L^2(\Omega)} + \| f\|_{H^{3/2}(\partial\Omega)} \right\}\\
 & \le C \e_n
 \left\{ 1+ L \e_{n-1}^{-1} + L^2 \e_n \e_{n-1}^{-2} \right\}
  \left\{ \| F\|_{L^2(\Omega)} + \| f\|_{H^{3/2}(\partial\Omega)} \right\}\\
&\le C (1+L)^2 \e_n \e_{n-1}^{-1} 
 \left\{ \| F\|_{L^2(\Omega)} + \| f\|_{H^{3/2}(\partial\Omega)} \right\}.
 \endaligned
 $$
This, together with (\ref{ind}), gives (\ref{tre1}).
\end{proof}

%%%%%%%%%%%%%%%%%%%%%%%

\section{Approximation}\label{section-5}

In preparation for the proofs of Theorems \ref{lipth} and \ref{blipth},
we establish several results on the approximation of solutions of $\mathcal{L}_\e (u_\e)=F$
by solutions of $\mathcal{L}_0 (u_0)=F$ in this section.
We start with a simple case, where $n=1$ and $A=A(x, y)$ is Lipschitz continuous in  $x$.

\begin{lemma}\label{lemma-5.1}
Suppose  $A=A(x, y)$ satisfies (\ref{elcon}) and is 1-periodic in $y$.
Also assume that $\|\nabla_x A\|_\infty<\infty$.
Let $\mathcal{L}_\e =-\text{\rm div} \big( A(x, x/\e)\nabla \big)$
and  $u_\e$ be a weak solution of $\mathcal{L}_\e (u_\e)=F$ in $B_{2r}=B(x_0, 2r)$,
where  $\e\le r\le 1$ and $F\in L^2(B_{2r})$.
Then there exists a weak solution to $\mathcal{L}_0 (u_0)=F$ in $B_r$ such that
\begin{equation}\label{5.01}
\aligned
& \left(\fint_{B_r} |u_\e -u_0 |^2 \right)^{1/2}\\
& \le C \left\{ \left(\frac{\e}{r}\right)^\sigma
+\e  \|\nabla_x A\|_\infty \right\}
\left\{
\left(\fint_{B_{2r}} |u_\e|^2\right)^{1/2}
+ r^2 \left(\fint_{B_{2r}} |F|^2 \right)^{1/2} \right\} ,
\endaligned
\end{equation}
where  $\sigma>0$ and $C$ depends only on $d$ and $\mu$.
\end{lemma}

\begin{proof}
By rescaling  we may assume  $r=1$.
To see this, we note that if $-\text{\rm div}  \big( A(x, x/\e)\nabla  u_\e\big) =F$ in $B_{2r}$
and $v (x)= u_\e (rx)$,
then $-\text{\rm div} \big( \widetilde{A} (x, x/\delta) \nabla v\big)=G$ in $B_2$,
where $\widetilde{A} (x, y)=A(rx, y)$, $\delta=\e/r$, and $G(x)=r^2 F(rx)$.
Also, observe that $\|\nabla_x \widetilde{A} \|_\infty= r \|\nabla_x A\|_\infty$.

Now, suppose that $-\text{\rm div}  \big( A(x, x/\e)\nabla  u_\e\big) =F$ in $B_{2}$.
Let $u_0\in H^1(B_{3/2})$  be the weak solution  to 
$$
\mathcal{L}_0 (u_0)=F \quad \text{ in }  B_{3/2}
\quad \text{ and } 
\quad 
u_0=u_\e \quad \text{ on } \partial B_{3/2}.
$$
Note that $u_0-u_\e\in H^1_0(B_{3/2})$ and
$$
\mathcal{L}_\e (u_0 -u_\e)=\text{\rm div} \big( (\widehat{A} -A^\e ) \nabla u_\e \big)
\quad \text{ in } B_{3/2}.
$$
It follows from the Meyers' estimates that
$$
\fint_{B_{3/2}} |\nabla (u_\e-u_0) |^q
\le C \fint_{B_{3/2}} |\nabla u_\e|^q
$$
for some $q >2$  and $C>0$, depending only on $d$ and $\mu$.
This, together with the Meyers' estimate,
$$
\left(\fint_{B_{3/2}} |\nabla u_\e|^q\right)^{1/q}
\le C \left(\fint_{B_2} |u_\e|^2\right)^{1/2}
+ C \left(\fint_{B_2} |F|^2\right)^{1/2},
$$
gives
\begin{equation}\label{5.02}
\left(\fint_{B_{3/2}} |\nabla u_0|^q\right)^{1/q}
\le 
C \left(\fint_{B_2} |u_\e|^2\right)^{1/2}
+ C \left(\fint_{B_2} |F|^2\right)^{1/2}.
\end{equation}
Also, by the  interior $H^2$ estimate for $\mathcal{L}_0$,
\begin{equation}\label{5.03}
\fint_{B(z,\rho)}
|\nabla^2 u_0|^2
\le C \fint_{B(z, 2\rho)} 
|F|^2
+ C \big( \|\nabla_x A\|_\infty^2 +\rho^{-2}
\big)
\fint_{B(z, 2\rho)} |\nabla u_0|^2,
\end{equation}
where $B(z, 2\rho)\subset B_2$,
we may deduce that 
\begin{equation}\label{5.04}
\aligned
\int_{B_{(3/2) -t}}
|\nabla^2 u_0|^2\, dx
 & \le C \int_{B_{3/2}}
|F|^2\, dx 
+ C \|\nabla_x A\|^2_\infty \int_{B_{3/2}} |\nabla u_0 |^2\, dx\\
&\qquad\qquad
+ C \int_{B_{(3/2)-(t/2)}}
\frac{ |\nabla u_0 (x) |^2\, dx }{|\text{\rm dist} (x, \partial B_{3/2})|^2 }
\endaligned
\end{equation}
for $0<t<1$.
By H\"older's inequality, the last term in the right-hand side of (\ref{5.04})
is bounded by
$$
C t^{-\frac{2}{q}-1} \left(\int_{B_{3/2}} |\nabla u_0|^q \right)^{2/q}.
$$
In view of (\ref{5.02}) and (\ref{5.04}) we obtain
\begin{equation}\label{5.05}
\int_{B_{(3/2) -t}}
|\nabla^2 u_0|^2\, dx
 \le C \left\{
t^{-\frac{2}{q}-1} + \|\nabla_x A\|^2_\infty\right\}
\left\{
\fint_{B_2} |F|^2
+\fint_{B_2} |u_\e|^2 \right\}
\end{equation}
for $0<t<1$, where $C$ depends only on $d$ and $\mu$.

Finally,  to finish the proof,
we use the estimate (\ref{L-2-estimate}) to obtain
$$
\aligned
\int_{B_{3/2}} |u_\e -u_0|^2
 & \le  C \e^2  (\|\nabla_x A\|^2 _\infty +1) \int_{B_{3/2}} |\nabla u_0|^2
+ C \e^2  \int_{B_{|x|<\frac{3}{2}-3\e}}
 |\nabla^2 u_0|^2\\
&\qquad\qquad
+ C \int_{\frac{3}{2}-4\e < |x|<\frac{3}{2}} |\nabla u_0|^2.
\endaligned
$$
We bound the second term in the right-hand side of  the inequality above by using (\ref{5.05}), and
the third term by using H\"older inequality and (\ref{5.02}).
It follows that
$$
\int_{B_{3/2}} |u_\e -u_0|^2
\le C \Big\{ \e^{1-\frac{2}{q}} 
+\e^2 \|\nabla_x A\|_\infty^2 \Big\}
\left\{ \int_{B_2} |u_\e|^2 +\int_{B_2} |F|^2 \right\}.
$$
This gives the estimate (\ref{5.01}) with $r=1$ and $\sigma = \frac12 -\frac{1}{q}>0$.
\end{proof}

The next lemma deals with the case $n=1$ and
$A=A(x, y)$ is H\"older continuous in $x$,
\begin{equation}\label{H-C}
|A(x, y)-A(x^\prime, y)|\le L |x-x^\prime|^\theta \quad \text{ for any } x, x^\prime \in \mathbb{R}^d,
\end{equation}
where $L\ge 0$ and $\theta\in (0,1)$.

\begin{lemma}\label{lemma-5.2}
Suppose  $A=A(x, y)$ satisfies (\ref{elcon}), (\ref{H-C}),  and is 1-periodic in $y$.
Let $\mathcal{L}_\e =-\text{\rm div} \big( A(x, x/\e)\nabla \big)$
and  $u_\e$ be a weak solution of $\mathcal{L}_\e (u_\e)=F$ in $B_{2r}=B(x_0, 2r)$,
where  $\e\le r\le 1$ and $F\in L^2 (B_{2r})$.
Then there exists a weak solution to $\mathcal{L}_0 (u_0)=F$ in $B_r$ such that
\begin{equation}\label{5.21}
\aligned
& \left(\fint_{B_r} |u_\e -u_0 |^2 \right)^{1/2}\\
& \le C \left\{ \left(\frac{\e}{r}\right)^\sigma
+\e^\theta L  \right\}
\left\{
\left(\fint_{B_{2r}} |u_\e|^2\right)^{1/2}
+ r^2 \left(\fint_{B_{2r}} |F|^2 \right)^{1/2} \right\} ,
\endaligned
\end{equation}
where  $\sigma>0$  depends only on $d$ and $\mu$.
The constant $C$ depends only on $d$, $\mu$ and $\theta$.
\end{lemma}

\begin{proof}
As in the proof of Lemma \ref{lemma-5.1},
by rescaling, we may assume $r=1$.
We also assume that $\e^\theta L<1$;
for otherwise the  inequality is  trivial.

By using a convolution in the $x$ variable  we may find a matrix $\widetilde{A}=\widetilde{A}(x, y)$ such that
$\widetilde{A}$ satisfies the ellipticity condition (\ref{elcon}), is 1-periodic in $y$, and 
\begin{equation}\label{5.22}
\| A-\widetilde{A} \|_\infty
\le C  L\e ^\theta
\quad
\text{ and } \quad
\|\nabla_x \widetilde{A}\|_\infty
\le C L \e ^{\theta-1},
\end{equation}
where  $C$ depends only on $d$ and $\theta$.
Let $v_\e$ be the weak solution to
\begin{equation}
-\text{\rm div} \big( \widetilde{A}(x, x/\e)\nabla v_\e\big) =F \quad
\text{ in } B_{3/2} \quad \text{ and } \quad
v_\e =u_\e \quad \text{ on } \partial B_{3/2}.
\end{equation}
By the energy estimate as well as the first inequality in (\ref{5.22}),
$$
\aligned
\fint_{B_{3/2}}
|\nabla (u_\e -v_\e)|^2
&\le C (L\e^\theta)^2
\fint_{B_{3/2}} |\nabla u_\e|^2\\
&\le C (L \e^\theta)^2 
\left\{
\fint_{B_2} |u_\e|^2 
+ \fint_{B_2} |F|^2 \right\},
\endaligned
$$
where we have used the Caccioppoli inequality for the last step.
This, together with Poincar\'e's  inequality, gives
\begin{equation}\label{5.23}
\left(\fint_{B_{3/2}}
| u_\e -v_\e|^2 \right)^{1/2}
\le C L \e ^\theta
\left\{
\left(\fint_{B_2} |u_\e|^2 \right)^{1/2}
+ \left(\fint_{B_2} |F|^2\right)^{1/2}\right\}.
\end{equation}
Next, we apply Lemma \ref{lemma-5.1} (and its proof) to the operator $-\text{\rm div} \big(\widetilde{A}(x, x/\e)\nabla \big)$.
Let $\widetilde{A}_0(x) $ denote the matrix of effective coefficients for $\widetilde{A}(x, y)$.
It follows that there exists $v_0 \in H^1(B_{5/4})$ such that 
$-\text{\rm div} \big( \widetilde{A}_0 (x)\nabla v_0)=F$ in $B_{5/4}$, and
\begin{equation}\label{5.24}
\aligned
\left(\fint_{B_{5/4}}
|v_\e -v_0|^2 \right)^{1/2}
 & \le C \big\{ \e^\sigma
+  \e^\theta  L \big\}
\left\{ \left(\fint_{B_{3/2}}
|v_\e|^2\right)^{1/2}
+ \left(\fint_{B_{3/2}} |F|^2 \right)^{1/2} 
\right\}\\
& \le C \big\{ \e^\sigma
+  \e^\theta  L \big\}
\left\{ \left(\fint_{B_{2}}
|u_\e|^2\right)^{1/2}
+ \left(\fint_{B_{2}} |F|^2 \right)^{1/2}
\right\},
\endaligned
\end{equation}
where we have used the second inequality in (\ref{5.22}) as well as (\ref{5.23}).

Finally, let $u_0$ be the weak solution to
$\mathcal{L}_0 (u_0)=F$ in $B_1 $ and $u_0=v_0$ on $\partial B_1$.
Observe that by the first inequality in (\ref{5.22}),
$$
\| \widetilde{A}_0  - \widehat{A} \|_\infty,
\le C  \e^\theta L, 
$$
where $C$ depends only on $d$ and $\mu$.
It follows that by Poincar\'e's inequality,
$$
\aligned
\int_{B_1} |u_0 -v_0|^2
&\le C 
\int_{B_1}|\nabla (u_0-v_0)|^2\\
&\le  C (\e^\theta L)^2
\int_{B_1} |\nabla v_0|^2\\
&\le C(\e^\theta L)^2
\left\{ \int_{B_{5/4}} |v_0|^2 +\int_{B_2} |F|^2 \right\}\\
&\le C (\e^\theta L)^2 
\left\{ \int_{B_2} |u_\e|^2 +\int_{B_2} |F|^2 \right\},
\endaligned
$$
where we have used Cacciopoli's inequality for the third inequality and (\ref{5.24}) for the fourth.
This, together with (\ref{5.23}) and \ref{5.24}), gives (\ref{5.21}) for $r=1$.
\end{proof}

We are now ready to handle the general case, where $n\ge 1$ and
\begin{equation}\label{op-5}
\mathcal{L}_\e=-\text{\rm div} \big( A(x, x/\e_1, \dots, x/\e_n)\nabla \big)
\end{equation}
with $0<\e_n<\e_{n-1}< \cdots< \e_1<1$.

\begin{theorem}\label{theorem-5.1}
Suppose that $A=A(x, y_1, \dots, y_n)$ satisfies conditions (\ref{elcon}), (\ref{pcon}), and
(\ref{lipcon}) for some $\theta\in (0, 1]$ and $L\ge 0$.
Let  $\mathcal{L}_\e $ be given by (\ref{op-5}) and $u_\e$  a weak solution of
$\mathcal{L}_\e (u_\e)=F$ in $B_{tr} =B (x_0, t r)$ for some $t>1$,
where $\e_1 \le r\le 1$ and $F\in L^2(B_{tr})$.
Then there exists $u_0\in H^1(B_r)$ such that $\mathcal{L}_0  (u_0)=F$ in $B_r$ and 
\begin{equation}\label{5.30}
\aligned
 \left(\fint_{B_r} |u_\e -u_0|^2 \right)^{1/2}
& \le C \left\{ \left(\frac{\e_1}{r} \right)^\sigma
+\left (\e_1 + \e_2/\e_1
+\cdots + \e_n /\e_{n-1} \right)^\theta L 
\right\}\\
& \qquad\qquad
\cdot \left\{ \left(\fint_{B_{tr}} |u_\e|^2 \right)^{1/2}
+ r^2 \left(\fint_{B_{tr}} |F|^2 \right)^{1/2} \right\},
\endaligned
\end{equation}
where $\sigma>0$ depends only on $d$ and $\mu$.
The constant $C$ depends only on $d$, $n$, $\mu$, $t$, and $\theta$.
\end{theorem}

\begin{proof}
We prove the theorem by an induction argument on $n$.
The case $n=1$  with $t=2$ is  given by Lemma \ref{lemma-5.2}.
The proof for the general case $t>1$ is similar.
Now suppose the  theorem is true for $n-1$.
To show it is true for $n$,
let $u_\e$ be a weak solution to $\mathcal{L}_\e (u_\e)=F$ in $B_{t r}$,
where $\mathcal{L}_\e$ is given by (\ref{op-5}).
Fix $\e>0$ and consider the matrix 
$$
E (x, y)=A(x, x/\e_1, \dots, x/\e_{n-1} , y)
$$
Note that $E$  satisfies the ellipticity condition (\ref{elcon}) and is 
1-periodic in $y$. Moreover, we have 
\begin{equation}\label{5.300}
| E (x, y) - E (x^\prime, y)|
\le C \e_{n-1}^{-\theta} L |x-x^\prime|^\theta \quad \text{ for any }
x, x^\prime \in \mathbb{R}^d,
\end{equation}
where $C$ depends only on $d$ and $n$.
Also recall that the matrix of effective coefficients for $E (x, y)$
is given by 
$$
A_{n-1} (x, x/\e_1 , \cdots, x/\e_{n-1} ),
$$
where $A_{n-1}(x, y_1, \cdots, y_{n-1})$ is given by (\ref{A-ell}) with $\ell =n$ and $A_n=A$.
Let $1<s<t$. 
By the theorem  for the case $n=1$,
 there exists $v_\e\in H^1(B_{s r})$  such that
$$
-\text{\rm div} \big( A_{n-1} (x, x/\e_1, \dots, x/\e_{n-1} )\nabla v_\e \big) =F \quad \text{ in } B_{ s r},
$$
and
\begin{equation}\label{5.31}
\aligned
\left(\fint_{B_{s r}}
|u_\e -v_\e|^2\right)^{1/2}
 & \le C \left\{ \left(\frac{\e_n }{r}\right)^\sigma
+ \left(\frac{\e_n}{\e_{n-1}} \right)^\theta L \right\}\\
&\qquad\qquad \cdot 
 \left\{ \left(\fint_{B_{t  r}} |u_\e|^2 \right)^{1/2}
+ r^2 \left(\fint_{B_{t  r}} |F|^2 \right)^{1/2} \right\}.
\endaligned
\end{equation}
By induction assumption there exists $u_0\in H^1(B_r)$ such that
$\mathcal{L}_0 (u_0)=F$ in $B_r$ and
\begin{equation}\label{5.32}
\aligned
 \left(\fint_{B_r} |v_\e -u_0|^2 \right)^{1/2}
& \le C \left\{ \left(\frac{\e_1 }{r} \right)^\sigma
+ (\e_1+ \e_2/\e_1
+\cdots + \e_{n-1}/\e_{n-2}  )^\theta L 
\right\}\\
& \qquad\qquad
\cdot \left\{ \left(\fint_{B_{s r}} |v_\e|^2 \right)^{1/2}
+ r^2 \left(\fint_{B_{s r}} |F|^2 \right)^{1/2} \right\}.
\endaligned
\end{equation}
Estimate (\ref{5.30}) follows readily from (\ref{5.31}) and (\ref{5.32}).
\end{proof}

\begin{remark}
{\rm
Let $\delta=\e_1  +\e_2/\e_1 +\cdots  + \e_{n-1}/\e_{n-2} $.
It follows from Theorem \ref{theorem-5.1} (with $t=2$)  that for $\delta\le r<1$,
\begin{equation}\label{remark-5.10}
\left(\fint_{B_r}
|u_\e -u_0|^2 \right)^{1/2}
\le C \left(\frac{\delta}{r } \right)^\sigma
\left\{ \left(\fint_{B_{2r}}
|u_\e|^2 \right)^{1/2} 
+ r^2 \left(\fint_{B_{2r}} |F|^2\right)^{1/2} \right\},
\end{equation}
where $\sigma>0$ depends only on $d$, $\mu$ and $\theta$.
The constant $C$ depends at most on $d$, $n$, $\mu$ and $(\theta, L)$.
Suppose that  $(\e_1, \e_2, \dots, \e_n)$ satisfies the condition (\ref{w-s-cond}).
Then $\delta \le C \e_1^\beta$ for some $\beta>0$ depending only on $n$ and $N$.
This, together with (\ref{remark-5.10}),  implies that for $\e_1\le r< 1$,
\begin{equation}\label{remark-5.1}
\left(\fint_{B_r}
|u_\e -u_0|^2 \right)^{1/2}
\le C \left(\frac{\e_1}{r } \right)^\rho
\left\{ \left(\fint_{B_{2r}}
|u_\e|^2 \right)^{1/2} 
+ r^2 \left(\fint_{B_{2r}} |F|^2\right)^{1/2} \right\},
\end{equation}
where $\rho>0$ depends only on $d$, $n$,  $\mu$, $\theta$, and $N$. 
}
\end{remark}

%%%%%%%%%%%%%%%%%%%%%%%%%%%%%%%%%%%%%%%%%%%

\section{Large-scale interior  estimates}\label{section-6}

This section focuses on  large-scale  interior estimates for $\mathcal{L}_\e (u_\e)=F$ and gives the proof of Theorem \ref{lipth}.
Throughout this section we assume that $\mathcal{L}_\e$ is given by (\ref{operator})
and  $A=A(x, y_1, \dots, y_n)$ satisfies (\ref{elcon}), (\ref{pcon}), and (\ref{lipcon}) for some $\theta\in (0, 1]$ and $L\ge 0$.
We also assume that $0<\e_n<\e_{n-1}<\dots< \e_1<1$ and  the condition (\ref{w-s-cond})
of well-separation  is satisfied.

We start with estimates of solutions of $\mathcal{L}_0 (u_0)=F$.
Let $\mathcal{P}$  denote  the set of linear functions.

\begin{lemma}\label{liple2}
  Let $u_0\in H^1(B_r)$ be a solution to $\mathcal{L}_0(u_0)=F$ in  $B_r=B(0, r)$, 
  where $0<r\leq 1$  and $ F\in L^p(B_r)$  for some  $p>d$. 
  Define
  \begin{align}\label{reliple2}
    G(r; u_0)=\frac{1}{r}\inf_{P\in \mathcal{P}}\left\{\left(\fint_{B_r}|u_0-P|^2\right)^{1/2}+  r^{1+\vartheta }|\na P|  \right\}
     +r \left(\fint_{B_r}|F|^p\right)^{1/p},
  \end{align}
  where $\vartheta=\min\{\theta, 1-d/p\}$.
  Then there exists $t \in(0, 1/8)$, depending only on $d$, $\mu$, $p$ and $(\theta, L)$ in \eqref{lipcon}, 
  such that
  $$
  G(t r; u_0)\leq \frac{1}{2}G(r; u_0). 
  $$
\end{lemma}

\begin{proof}
Let  $ P_0= x \cdot  \nabla  u_0(0)  +u_0(0)$.
Then
  \begin{equation}
   \label{pliplem0}
   \aligned
    G(tr; u_0)&\leq \frac{1}{tr}   \|u_0-P_0\|_{L^\infty(B_{tr})}  +tr\left(\fint_{B_{tr}}| F |^p\right)^{1/p}+(tr)^\vartheta |\na u_0(0)|\\
    &\leq (tr)^\vartheta\| \na u_0\|_{C^{0,\vartheta}(B_{tr})}  +tr\left(\fint_{B_{tr}}|F |^p\right)^{1/p}+(tr)^\vartheta | \na u_0(0)|  \\
    &= (tr)^\vartheta\| \na (u_0-P)\|_{C^{0,\vartheta}(B_{tr})}  +tr\left(\fint_{B_{tr}}|F|^p\right)^{1/p}+(tr)^\vartheta |\na u_0(0)|
  \endaligned
  \end{equation}
  for any $P\in \mathcal{P}$.
  Note that
  \begin{equation}
  \label{pliplem1}
  \aligned
  tr\left(\fint_{B_{rt}}|F|^p\Big)^{1/p}\leq C  t^{1-d/p} r \Big(\fint_{B_{r}}|F|^p\right)^{1/p}.
  \endaligned
  \end{equation}
  By interior Lipschitz estimates for $u_0$, we may deduce that
 \begin{equation}\label{pliplem2}
 \aligned
 |\na u_0(0)| &\leq \frac{C}{r} \left(\fint_{B_r} |u_0-b|^2\right)^{1/2} +Cr\left(\fint_{B_r}|F|^p\right)^{1/p} \\
 &\leq \frac{C}{r} \left(\fint_{B_r} |u_0-P|^2\Big)^{1/2}+   \frac{C}{r}\Big(\fint_{B_r} |P-b|^2\right)^{1/2}  +Cr\left(\fint_{B_r}|F|^p\right)^{1/p}\\
 &\leq  \frac{C}{r} \left(\fint_{B_r} |u_0-P|^2\right)^{1/2} +C |\na P|+ Cr\left(\fint_{B_r}|F|^p\right)^{1/p},
 \endaligned
 \end{equation}
 where $b= P(0) $.
Also, note that
 \begin{align*}
-\text{div} \big(\widehat{A}  \na (u_0-P)\big)= F +\text{div} \big([\widehat{A} -\widehat{A} (0)]  \na P\big) \quad \text{ in } B_r.
\end{align*}
By $C^{1,\vartheta}$ estimates for the elliptic operator $\mathcal{L}_0$,  we obtain that for $0<t<1/2$,
\begin{equation}\label{pliplem3}
\aligned
 \|\na (u_0-P)\|_{C^{0,\vartheta}(B_{tr})}
 &\leq \|\na (u_0-P)\|_{C^{0,\vartheta}(B_{r/2})} \\
  &\leq \frac{C}{r^{1+\vartheta}} \Big( \fint_{B_r} |u_0-P|^2\Big)^{1/2}  +C r^{-\vartheta}\| [\widehat {A}-\widehat{A}(0)] \na P \|_{L^\infty(B_r)} \\
  &\quad+  C  \| [\widehat{A}-\widehat{A}(0)] \na P \|_{C^{0,\vartheta}(B_r)}+   Cr^{1-\vartheta} \Big(\fint_{B_r}|F|^p\Big)^{1/p} \\
&\leq \frac{C}{r^{1+\vartheta}} \Big( \fint_{B_r}  |u_0-P|^2\Big)^{1/2}  + C |\na P | +   Cr^{1-\vartheta} \Big(\fint_{B_r}|F|^p\Big)^{1/p}.
\endaligned
\end{equation}
By using  \eqref{pliplem1}--\eqref{pliplem3} to bound the right-hand side of \eqref{pliplem0}, it yields that
$$ G(tr; u_0) \leq C t^\vartheta G(r; u_0)$$
for some constant $C$ depending only on $d$,  $\mu$, $p$ and $(\theta, L)$ in \eqref{lipcon}. 
The desired result  follows by  choosing  $t$  so small that $Ct^\vartheta\le 1/2$.
\end{proof}

\begin{lemma}\label{liple3}
  Let $u_{\va}\in H^1(B_1)$ be a solution of $\mathcal{L}_{\va} (u_{\va})=F$ in $B_1$, 
  where  $F\in L^p(B_1)$ for some $p>d$ and $\varepsilon\in (0, 1/4)$.
   For $0<r\leq 1$, we define
  \begin{align}
  \label{defh}
  \begin{split}
  &H(r)=\frac{1}{r}\inf_{P\in \mathcal{P} }\left\{\left(\fint_{B_r}|u_{\va}-P|^2\right)^{1/2} 
  +r^{1+\vartheta} |\na P| \right\} 
   +r\left(\fint_{B_r}|F|^p\right)^{1/p},\\
  &\Phi(r)=\frac{1}{r}\inf_{b\in \mathbb{R}} \left(\fint_{B_r}|u_{\va}-b|^2\right)^{1/2}+r\left(\fint_{B_r}|F |^2\right)^{1/2}.
  \end{split}
  \end{align}
    Let $t \in(0, 1/8)$ be given by Lemma \ref{liple2}.
    Then for  $r\in (\e_1, 1/2]$,
    \begin{align}\label{liple3re}
      H(t r)\leq \frac{1}{2}H(r)+C\left(\frac{\e_1 }{r}\right)^\rho \Phi(2 r),
    \end{align}
     where $\rho>0$ and  $C$ depends at most on $d$, $n$, $\mu$, $p$, $(\theta, L)$ in \eqref{lipcon}, 
     and $N$ in (\ref{w-s-cond}).
  \end{lemma}
  
  \begin{proof}
    For any fixed $r\in (\e_1, 1/2 ]$, let $u_0$ be the solution to $\mathcal{L}_0 (u_0)=F$ in $B_r$,  given in Theorem  \ref{theorem-5.1}. 
    By the definitions of $G, H$ and $\Phi$, we have
    \begin{align*}
      H(t r)&\leq \frac{1}{t r}\left(\fint_{B_{t r}}|u_{\va}-u_0|^2\right)^{1/2}+G(t r; u_0)\\
      &\leq \frac{1}{t r}\left(\fint_{B_{t r}}|u_{\va}-u_0|^2\right)^{1/2}+\frac{1}{2}G(r; u_0)\\
       &\leq \frac{C}{ r}\left(\fint_{B_{r}}|u_{\va}-u_0|^2\right)^{1/2} +\frac{1}{2}H(r)\\
      &\leq C\left(\frac{ \e_1 }{r}\right)^\rho
      \left\{\left( \frac{1}{r} \fint_{B_{2 r}}|u_{\va}-b|^2\right)^{1/2}+r \left(\fint_{B_{2r}} | F |^2\right)^{1/2}\right\}+\frac{1}{2}H(r)
    \end{align*}
for any $b\in \mathbb{R}$, where we have used Lemma \ref{liple2} and (\ref{remark-5.1})  in the second and last inequalities, respectively.
\end{proof}

The following lemma can be found in \cite[p.155]{shenan2017}.

\begin{lemma}\label{liple4}
  Let $H(r)$ and $h(r)$ be two nonnegative continuous functions on the interval $(0,1]$ and let $t \in (0,1/4)$. Assume that
\begin{align}\label{Lip_cond_1}
\max_{r\leq t\leq 2r} H(t)\leq C_0H(2r), ~~~~~\max_{r\leq t, s\leq 2r} | h(t)-h(s)|\leq C_0H(2r),
\end{align}
for any $r\in [\de, 1/2]$, and also
\begin{align}\label{Lip_cond_2}
H( t  r) \leq \frac{1}{2} H(r) + C_0 \omega (\de/r)\left\{ H(2r)+h(2r)\right\},
\end{align}
for any $r\in [\de, 1/2]$, where $\omega$ is a nonnegative increasing function on $[0, 1]$ such that
$\omega(0)=0$ and
\begin{align}\label{Lip_cond_3}
\int_0^1 \frac{\omega(s)}{s} ds<\infty.
\end{align}
Then
\begin{align}\label{Lip_es_H}
\max_{\de\leq r\leq 1} \left\{H(r)+h(r)\right\}\leq C \left\{H(1) +h(1)\right\},
\end{align}
where $C$ depends only on $C_0$, $\theta_0$ and $\omega$.
\end{lemma}

The next lemma gives the large-scale Lipschitz estimate down  to the scale $\e_1$.

\begin{lemma}\label{th41}
 Let $u_{\va}\in H^1(B_1)$ be a solution to $\mathcal{L}_{\va} (u_{\va})=F$ in $B_1$, 
 where $B_1=B(x_0,1) $ and $F\in L^p(B_1)$ for some $p>d\ge 2$. 
 Then for $\e_1  \leq  r< 1$, 
\begin{align}\label{reth41}
\left(\fint_{B_r} |\na u_{\va}|^2\right)^{1/2}\leq C \left\{ \left(\fint_{B_1}
|  \nabla u_{\va}|^2\right)^{1/2} +
\left(\fint_{B_1} |F  |^p\right)^{1/p}\right\},
\end{align}
where $C$ depends only on $d$, $n$, $\mu$, $p$,  $(\theta, L)$ in \eqref{lipcon}, and $N$ in (\ref{w-s-cond}).
\end{lemma}

\begin{proof}
By translation we may assume $x_0=0$. 
Let $P_r, b_r$ be a linear function and constant  achieving the infimum in \eqref{defh}.
 In  particular, 
 $$
 H(r)=\frac{1}{r}
 \left(\fint_{B_r}|u_{\va}-P_r|^2\right)^{1/2}+r^\vartheta 
 |\na P_r| +r\Big(\fint_{B_r}|F|^p\Big)^{1/p}.
 $$ 
 Let $h(r)=|\nabla P_r|$.
It  follows by Poincar\'{e}'s inequality that
\begin{align*}
  \Phi(2r)\leq H(2r)+\frac{1}{r}\inf_{b\in \mathbb{R }} \left( \fint_{B_{2r}}|P_{2r}-b|^2\right)^{1/2}\leq H(2r)+Ch(2r).
  \end{align*}
This, combined with  \eqref{liple3re}, gives 
 \eqref{Lip_cond_2} with $\omega(t)=t^\rho$, which satisfies \eqref{Lip_cond_3}.

For $t\in[r, 2r]$, it is obvious that $H(t)\leq CH(2r)$. Furthermore, observe that
\begin{align*}
|h(t)-h(s)| &= |\nabla (P_t-P_s)|\leq \frac{C}{r} \left( \fint_{B_r}|P_t-P_s|^2\right)^{1/2}\\ 
&\leq \frac{C}{r} \left( \fint_{B_r}|u_{\va}-P_t|^2\right)^{1/2}+\frac{C}{r} \left( \fint_{B_r}|u_{\va}-P_s|^2\right)^{1/2}\\
&\leq \frac{C}{t} \left( \fint_{B_r}|u_{\va}-P_t|^2\right)^{1/2}+\frac{C}{s} \left( \fint_{B_s}|u_{\va}-P_s|^2\right)^{1/2}\\
&\leq C\{H(t)+H(s)\} \\
&\leq CH(2r)
\end{align*}
 for all $t, s\in[r, 2r]$,
which is exactly the condition \eqref{Lip_cond_1}.

Thanks to \eqref{Lip_es_H},  we obtain that
\begin{align}
  \frac{1}{r}\inf_{b\in \mathbb{R}}\left( \fint_{B_r}|u_{\va}-b|^2\right)^{1/2}&\leq H(r)+\frac{1}{r}\inf_{b\in \mathbb{R}}\left( \fint_{B_r}|P_r-b|^2\right)^{1/2}\nonumber\\
  &\leq C\{H(r)+h(r)\}\nonumber\\
  &\leq C\{H(1)+h(1)\}\nonumber\\
  &\leq C\left\{\left( \fint_{B_1}|u_{\va}|^2\right)^{1/2}+ \left( \fint_{B_1}|F|^p\right)^{1/p}\right\},\label{preth411}
\end{align}
for any $r\in [\e_1, 1/2]$, where for the last step the following observation is used,
\begin{align*}
  h(1)&\leq C\left( \fint_{B_1}|P_1|^2\right)^{1/2}\\&\leq C\left(\fint_{B_1}|u_{\va}-P_1|^2\right)^{1/2}+C\left(\fint_{B_1}|u_{\va}|^2\right)^{1/2}\\
  &\leq CH(1)+C\left(\fint_{B_1}|u_{\va}|^2\right)^{1/2}.
\end{align*}
 The estimate \eqref{reth41} follows readily  from \eqref{preth411} by Poincar\'e and Caccioppoli's inequalities.
\end{proof}

 We are now ready to prove Theorem \ref{lipth}.
 
\begin{proof}[\textbf{Proof of Theorem \ref{lipth}}] 

 The proof uses an induction on $n$ and relies on  Lemma \ref{th41} and a rescaling argument.
 The case $n=1$ follows directly from Lemma \ref{th41} by translation and dilation.
 Assume the theorem is true for $n-1$. 
 Suppose 
 $$
 \text{\rm div} \big( A(x, x/\e_1, \dots, x/\e_n) \nabla u_\e \big)
 =F \quad \text{ in } 
 B_R=B(x_0, R)
 $$
  for some $0<R\le 1$.
 We need to show that
 \begin{equation}\label{Lip-6}
 \left(\fint_{B_r} |\nabla u_\e |^2 \right)^{1/2}
 \le C \left\{  \left(\fint_{B_R} |\nabla u_\e |^2 \right)^{1/2}
 +  R \left(\fint_{B_R} |F |^p \right)^{1/p}\right\},
\end{equation}
for $\e_n  \le r<R\le 1$.
By translation and dilation we may assume that $x_0=0$ and $R=1$.
Note that the case $(1/8)\le r< R=1$ is trivial.
If $\e_n < r\le (1/8)$, we may cover the ball $B(0, r)$ with a finite number of
balls $B(x_\ell, \e_n )$, where $x_\ell\in B(0, r)$.
Consequently, it suffices to prove (\ref{Lip-6}) for the case $r=\e_n $ and $R=1$.
We further note that by Lemma \ref{th41},
the estimate (\ref{Lip-6}) holds for $r=\e_1$ and $R=1$.

To reach the finest scale $\e_n$, we let  $w(x)= u_\e (\e_1  x)$.
Then 
$$
-\text{\rm div} \big( E(x, x/(\e_2\e_1^{-1}), \dots, x/(\e_n \e_1^{-1})) \nabla w\big)
=H \quad \text{ in } B_1,
$$
where $H(x)=\e_1^{2}  F(\e_1 x)$ and
$$
 E(x, y_2, \dots, y_n) =A(\e_1 x, x, y_2, \dots, y_n).
$$
Observe that the matrix $E$ satisfies (\ref{elcon}) and is 1-periodic in $(y_2, \dots, y_n)$.
It also satisfies the smoothness condition (\ref{lipcon})
 with the same constants $\theta$ and $L$ as for $A$.
 Furthermore, the $(n-1)$ scales $(\e_2\e_1^{-1}, \dots, \e_n \e_1^{-1})$ 
 satisfies the condition (\ref{w-s-cond}) of well-separation.
 Thus,  by the induction assumption, 
\begin{equation}\label{Lip-7}
 \left(\fint_{B_r} |\nabla w|^2 \right)^{1/2}
 \le C \left\{  \left(\fint_{B_1} |\nabla w|^2 \right)^{1/2}
 + \left(\fint_{B_1} |H|^p \right)^{1/p} \right\},
 \end{equation}
 for $r=\e_n/\e_1 $.
By a change of variables it follows that
(\ref{Lip-6}) holds for $r=\e_n $ and $R=\e_1$. 
This, combined with the inequality   for $r=\e_1$ and $R=1$, implies that (\ref{Lip-6}) holds for
$r=\e_n $ and $R=1$.
The proof is complete.
\end{proof}

\begin{remark}\label{remark-6.1}
{\rm
It follows from the proof of Theorem \ref{lipth} that
without the condition (\ref{w-s-cond}), the estimate (\ref{relipth}) continues to
hold if
$$
\e_1 +\left( \e_2/\e_1 + \cdots + \e_{n}/\e_{n-1} \right)^N
\le r< R \le 1,
$$
for any $N\ge 1$.
In this case the constant $C$  in (\ref{relipth}) also depends on $N$.
The case $N=1$ follows by using (\ref{remark-5.10}) in the place of (\ref{remark-5.1}).
The general case is proved by an induction argument on $N$.
Suppose the claim is true for some $N\ge 1$.
Assume that $\beta=\e_2/\e_1+\cdots + \e_n/\e_{n-1}\ge \e_1$ (for otherwise, there is nothing to prove).
Let $w(x)=u_\e (\beta x)$. Then 
$-\text{\rm div} \big( E(x, x/(\beta^{-1} \e_1 ), \dots, x/(\beta^{-1}\e_n ) ) \nabla w \big) =H$,
where $E(x, y_1, \dots y_n)=A(\beta x, y_1, \dots, y_n)$.
By the induction assumption, the inequality (\ref{Lip-7}) holds for $\beta^{-1} \e_1 +\beta^N<r<1$.
By a change of variables we obtain (\ref{relipth}) for 
$\e_1 +\beta^{N+1}  \le r< R=\beta$.
This, together with the estimate for the case $N=1$, gives (\ref{relipth})
for $\e_1 +\beta^{N+1}\le r< R\le 1$.
}
\end{remark}

 %%%%%%%%%%%%%%%%%%%%%%%%%%%

 %%%%%%%%%%%%%%%%%%%%%%%%%%%
 
\section{Large-scale boundary Lipschitz estimates}\label{section-7}

This section is devoted to  the large-scale boundary Lipschitz estimate  and contains 
the proof of  Theorem \ref{blipth}.
Throughout the section we assume that $\mathcal{L}_\e$ is given by (\ref{operator}) with 
$A=A(x, y_1, \dots, y_n)$ satisfying conditions  (\ref{elcon}), (\ref{pcon}) and (\ref{lipcon}) for some $0< \theta \le 1$.
The condition (\ref{w-s-cond}) is also imposed. 

Let $\psi: \R^{d-1}\rightarrow \R$ be a $C^{1,\alpha }$ function with
 \begin{align} \label{cbd}
 \psi(0)=0 \quad \text{ and } \quad 
  \|\nabla \psi \|_\infty + \|\na\psi\|_{C^{0,\alpha}(\R^{d-1})}\leq M.
 \end{align}
Set
\begin{align}\label{drder}
\begin{split}
&Z_r=Z(r,\psi)=\left\{ (x',x_d)\in \mb{R}^d: |x'|<r \text{ and } \psi(x')<x_d< 10(M+10) r \right\},\\
&I_r=I (r,\psi)=\left\{ (x',\psi(x'))\in \mb{R}^d: |x'|<r \right\}.
\end{split}
\end{align}
For $f\in C^{1,\alpha}(I_r)$ with $0<\alpha<1$,
 we introduce  a scaling-invariant norm, 
 \begin{equation}\label{C-norm}
  \|f\|_{{C}^{1,\alpha}(I_r)}=\|f\|_{L^\infty(I_r)} + r\|\na_{\tan}  f\|_{L^\infty(I_r)}+ r^{1+\alpha} \|\na_{\tan}  f\|_{C^{0,\alpha}(I_r)},
  \end{equation}
where  $\nabla_{\tan} f$ denotes the tangential gradient of $f$ and 
$$ 
\| g\|_{C^{0,\alpha}(I_r)}= \sup_{x,y\in I_r, x\neq y} \frac{|g(x)-g (y)|}{|x-y|^\alpha}.
$$

\begin{theorem}\label{th51}
Let $u_{\va}\in H^1(Z_R)$ be a weak solution of  $\mathcal{L}_{\va} ( u_{\va}) =F$ in $Z_R$ 
and $u_{\va}=f$ on $I_R$, where $0<\e_n <R\leq 1$, $F\in L^p(Z_R)$ for some $p>d$,  and 
$f\in C^{1,\alpha}(I_R)$. Then for $\va_n \leq r< R$, 
\begin{align}\label{reth51}
\left(\fint_{Z_r} |\na u_\varepsilon|^2\right)^{1/2}\leq C \left\{ \left(\fint_{Z_R}
| \nabla  u_\varepsilon|^2\right)^{1/2} +  R^{-1}  \|f\|_{C^{1,\alpha}(I_R)} +
R\left(\fint_{Z_R} |F |^p\right)^{1/p}\right\},
\end{align}
where $C$ depends at most  on $d$, $n$, $\mu$, $p$,  $(\theta, L)$ in \eqref{lipcon},   $N$ in 
(\ref{w-s-cond}), and $(\alpha, M)$  in \eqref{cbd}.
\end{theorem}

Theorem \ref{blipth} follows readily from Theorem \ref{th51}
by translation and a suitable  rotation of the coordinate system.
To prove Theorem \ref{th51}, we use the same approach as in 
 the proof of Theorem \ref{th41}. 
 We will  provide only a sketch of the proof for Theorem \ref{th51}.
 
 First, we point out that  the rescaling argument, which is used extensively for interior estimates,
 works equally well in the case of boundary estimates.
 Indeed, suppose $\mathcal{L}_\e (u_\e)=F $ in $Z(r,  \psi) $ and
 $u_\e=f$ on $I  (r, \psi) $ for some $0<r \le 1$.
 Let $v(x)=u_\e (rx)$.
 Then 
 $$
 -\text{\rm div} \big( \widetilde{A} (x, x/\e_1 r^{-1}, \dots, x/\e_n r^{-1}) \nabla v \big) =G \quad \text{ in }
 Z(1, \psi_r)
\quad 
\text{ and}
\quad
v=g \quad  \text{ on } I (1, \psi_r),
$$
where $\widetilde{A} (x, y_1, \dots, y_n)=A(rx, y_1, \dots, y_n)$, 
$G(x)=r^2 F(rx)$, $g(x)= f(rx)$, and $\psi_r (x^\prime) =r^{-1} \psi(rx^\prime)$.
 Since $\nabla \psi_r (x^\prime) =\nabla \psi (rx^\prime)$ and $0<r\le 1$,
 the function $\psi_r$ satisfies the condition (\ref{cbd}) with the same $M$.
 Also, note that  $\| f\|_{C^{1, \alpha} (I  (r, \psi))}= \| g\|_{C^{1, \alpha} (I  (1, \psi_r))}$.
 As a result, it suffices to prove Theorem \ref{th51} for $R=1$.
 
Next, we establish an approximation result in the place of (\ref{remark-5.1}).
Define
$$
\| f\|_{C^1(I_r)}
=\|f\|_{L^\infty(I_r)} +  r \|\nabla_{\tan} f\|_{L^\infty (I_r)}.
$$

\begin{theorem}\label{apth2}
Let $u_{\va}\in H^1(Z_{2r}) $ be a  weak  solution of $\mathcal{L}_{\va} (u_{\va})=F$ in $Z_{2r}$
 and $u_{\va}=f$ on $I_{2r}$, where $0< \e \le   r\leq 1 $. 
 Then there exists $u_0\in H^1(Z_r)$ such that $\mathcal{L}_0 (u_0)=F$ in $Z_r$,
  $u_0=f$ on $I_{r}$, and
  \begin{equation}\label{reapth2}
  \aligned
    &\left(\fint_{Z_r}|u_{\va}-u_0|^2\right)^{1/2}\\
    &\leq C \left(\frac{\e_1}{r}\right)^\rho  \left\{\left(\fint_{Z_{2r}}|u_{\va} |^2\right)^{1/2}+r^2\left(\fint_{Z_{2r}}| F |^2\right)^{1/2}+ \|f\|_{C^{1}(I_{2r})}\right\}.
  \endaligned
  \end{equation}
 The constants $\rho\in (0, 1)$ and $C>0$  depend at most  on $d$,  $n$,  $\mu$, $(\theta, L)$ in \eqref{lipcon},
 $N$ in  (\ref{w-s-cond}), and $(\alpha, M)$ in \eqref{cbd}.
\end{theorem}

\begin{proof}
The proof of (\ref{reapth2}) is similar to that of (\ref{remark-5.1}).

\noindent{Step 1.}\ \
Assume  that  $n=1$,
$\mathcal{L}_\e =-\text{\rm div} \big(A(x,  x/\e)\nabla \big)$ and  $A(x, y)$ is Lipschitz continuous in $x$.
Suppose that $\mathcal{L}_\e(u_\e)=F$ in $Z_{2r}$ and $u_\e=f$ on $I_{2r}$.
Show that there exists $u_0\in H^1(Z_r)$ such that $\mathcal{L}_0 (u_0)=F$ in $Z_r$,
$u_0=f$ on $I_r$, and 
\begin{equation}\label{5.01-b}
\aligned
& \left(\fint_{Z_r} |u_\e -u_0|^2\right)^{1/2}\\
&\le C \left\{ \left(\frac{\e}{r} \right)^\sigma
+\e \|\nabla_x A\|_\infty \right\}
\left\{ \left(\fint_{Z_{2r}} |u_\e|^2\right)^{1/2} 
+ r^2 \left(\fint_{Z_{2r}} |F|^2\right)^{1/2} +\| f\|_{C^1( I_{2r}) }  \right\}.
\endaligned
\end{equation}

The proof of (\ref{5.01-b}) is similar to (\ref{5.01}).
By rescaling we may assume $r=1$.
Let $u_0$ be the weak solution of
$$
\mathcal{L}_0(u_0)=F \quad \text{ in } \Omega \quad \text{ and } \quad u_0=u_\e \quad \text{ on } \partial \Omega,
$$
where $\Omega=Z_{3/2}$.
By using (\ref{L-2-estimate}), we obtain 
$$
\int_\Omega |u_\e -u_0|^2
\le C \e^2 (\|\nabla_x A\|_\infty^2 +1) \int_\Omega |\nabla u_0|^2
+ C \e^2 \int_{\Omega\setminus \Omega_{3\e}} |\nabla^2  u_0|^2
+ C \int_{\Omega_{4\e} } |\nabla u_0|^2.
$$
The rest of the proof is the same as the proof of Lemma \ref{lemma-5.1}, using interior $H^2$ estimates for $\mathcal{L}_0$
as well as Meyers' estimates for $\mathcal{L}_\e$,
\begin{equation}\label{Meyer-7}
\left(\fint_{Z_{3/2}} |\nabla u_\e|^q \right)^{1/q}
\le C \left\{ \left(\fint_{Z_2} |u_\e|^2\right)^{1/2}
+ \| f\|_{C^1(I_2)}
+ \left(\fint_{Z_2} |F|^2\right)^{1/2} \right\}
\end{equation}
for some $q>2$, depending only on $d$, $\mu$ and $M$.

\medskip

\noindent{Step 2.}\ \ 
Assume $n=1$ and $A(x, y)$ is H\"older continuous in $x$.
Suppose $\mathcal{L}_\e(u_\e)=F$ in $Z_{2r}$ and $u_\e=f$ on $I_{2r}$.
Show that there exists $u_0\in H^1(Z_r)$ such that $\mathcal{L}_0 (u_0)=F$ in $Z_r$,
$u_0=f$ on $I_r$, and 
\begin{equation}\label{5.21-b}
\aligned
& \left(\fint_{Z_r} |u_\e -u_0|^2\right)^{1/2}\\
&\le C \left\{ \left(\frac{\e}{r} \right)^\sigma
+\e^\theta L \right\}
\left\{ \left(\fint_{Z_{2r}} |u_\e|^2\right)^{1/2} 
+ r^2 \left(\fint_{Z_{2r}} |F|^2\right)^{1/2} +\| f\|_{C^1( I_{2r}) }  \right\}.
\endaligned
\end{equation}
As in the case of (\ref{5.21}), the estimate (\ref{5.21-b}) follows from (\ref{5.01-b}) 
by approximating $A(x, y)$ in the $x$ variable.

\medskip

\noindent{Step 3.}
As in the interior case,
the case $n>1$ follows from (\ref{5.21-b}) by an induction argument on $n$.
\end{proof}

The following two lemmas will be used in the place of  Lemmas \ref{liple2} and \ref{liple3}.
Recall that $\mathcal{P}$ denotes the set of linear functions in $\mathbb{R}^d$.

\begin{lemma}\label{bll2}
  Let $u_0\in H^1(Z_r)$ be a weak solution of
 $  \mathcal{L}_0 (u_0)  =F  \text{ in } Z_r$ and $u_0=f$ on $I_r$,
   where $0<r\leq 1, F\in L^p(Z_r)$ for some  $p>d$,  and $f\in C^{1,\alpha}(I_r)$ 
   for some $0<\alpha<1$. Define 
  $$
  \aligned
   \mathcal{ G }(r; u_0)&= \inf_{P\in \mathcal{P}}\frac{1}{r}
   \left\{\left(\fint_{Z_r}|u_0-P|^2\right)^{1/2} +r^{1+\vartheta} | \na P| 
  + \|f-P\|_{C^{1,\alpha}(I_r)} \right\} \\
  &\qquad
    +r \left(\fint_{Z_r}|F|^p\right)^{1/p},
    \endaligned
    $$
    where $\vartheta=\min\{\theta,\alpha, 1-d/p\}.$
  Then there exists $t \in(0, 1/8)$, depending only on $d$, $n$,  $\mu$, $p$,  $(\theta, L)$ in \eqref{lipcon}, and $(\alpha, M)$ in \eqref{cbd}, such that,
$$\mathcal{G}(t r; u_0)\leq \frac{1}{2}\mathcal{G}(r; u_0). $$
\end{lemma}

\begin{proof}
The proof is similar to that of Lemma \ref{liple2}.
Let  $P_0(x)= \nabla u_0(0)\cdot x+u_0(0)$. Then for $0<t< (1/8)$,
  \begin{equation}\label{pbll21}
  \aligned
   \mathcal{ G}(tr; u_0)
     &\leq  C (tr)^\vartheta  \big\{\|\nabla u_0\|_{C^{0,\vartheta}(  Z_{tr} )}
     +  | \na u_0(0)|   \big\} + tr \left(\fint_{  Z_{tr} }|F|^p\right)^{1/p}\\
     &\le C (tr)^\vartheta  \big\{\|\nabla (u_0-P)\|_{C^{0,\vartheta}(  Z_{tr} )}
     +  | \na u_0(0)-\nabla P | +|\nabla P|    \big\}\\
     &\qquad\qquad\qquad
      + tr \left(\fint_{  Z_{tr} }|F|^p\right)^{1/p}
     \endaligned
     \end{equation}
  for any $P\in \mathcal{P}$, where we have used the fact $\nabla P$ is constant.
  Note that 
 \begin{align*}
-\text{div} \big(\widehat{A}   \na (u_0-P)\big)= F  +\text{div} \big([\widehat{A}-\widehat{A} (0)]  \na P\big) ~\text{ in } Z_r.
\end{align*}
By boundary $C^{1,\vartheta}$ estimates for the operator $\mathcal{L}_0$  in $C^{1,\alpha}$ domains, 
it follows that  for $0<t<(1/8)$,
\begin{equation}\label{pbll22}
\aligned
 & \|\na (u_0-P)\|_{C^{0,\vartheta}(Z_{tr})} +\| \nabla (u_0 -P)\|_{L^\infty(Z_{tr})}\\
 &
 \leq   \|\na (u_0-P)\|_{C^{0,\vartheta}(Z_{r/2})} + \| \nabla (u_0 -P)\|_{L^\infty(Z_{r/2})}\\
&
\leq  \frac{C}{r^{1+\vartheta}}
  \left( \fint_{Z_r} |u_0-P|^2\right)^{1/2}  +  C  |\nabla P| + Cr^{1-\vartheta}  \left(\fint_{Z_r}|F|^p\right)^{1/p}
+ \frac{C}{r^{1+\vartheta}}
 \| f-P\|_{C^{1, \alpha} (I_r)}
 \endaligned
\end{equation}
for any $P\in \mathcal{P}$.
This, together with (\ref{pbll21}), implies that $\mathcal{G}(tr; u_0)\le C t^{\vartheta} \mathcal{G} (r; u_0)$.
To complete the proof, we choose $t$ so small that $Ct^{\vartheta}\le (1/2)$.
\end{proof}

\begin{lemma}\label{bll3}
  Let $u_{\va}\in H^1(Z_1)$ be a weak  solution of $\mathcal{L}_{\va} (u_{\va}) =F$ in $Z_1$ and
  $u=f$ on $I_1$,
  where  $0<\e< (1/4)$,
  $F\in L^p(Z_1)$ for some $p>d$ and $f\in C^{1, \alpha}(I_1)$ for some $\alpha>0$.
   For $0<r\leq 1$, define
  \begin{align}
  \label{defhh}
  \begin{split}
   \mathcal{H}(r) &=\inf_{P\in \mathcal{P}}\frac{1}{r}
   \left\{ \left(\fint_{Z_r}|u_{\va}-P|^2\right)^{1/2} +r^{1+\vartheta} |\na P|
 + \|f-P\|_{{C}^{1,\alpha}(I_r)} \right\}\\
 &\qquad\qquad\qquad\qquad
  +r \left(\fint_{Z_r}|F|^p\right)^{1/p},\\
\Upsilon(r)=\inf_{b\in \mathbb{R} } \ &\frac{1}{r} 
\left\{\left(\fint_{Z_r}|u_{\va}-b|^2\right)^{1/2}+ \|f-b\|_{{C}^{1,\alpha}(I_r)}\right\}
+r\left(\fint_{Z_r}|F|^2\right)^{1/2}.
  \end{split}
  \end{align}
    Let $t \in(0, 1/8)$ be given by Lemma \ref{liple2}. Then for any $r\in [\e_1, 1/2]$,
    \begin{align}\label{bll3re}
     \mathcal{ H}(t r)\leq \frac{1}{2}\mathcal{H}(r)+C\left(\frac{\e_1}{r}\right)^\rho \Upsilon(2r),
    \end{align}
    where $\rho>0$ and $C>0$ depends at most on $d$, $n$,  $\mu$, $p$,   $(\theta, L)$ in \eqref{lipcon}, $N$ in (\ref{w-s-cond}), and $(\alpha, M)$ in (\ref{cbd}).
  \end{lemma}
  
  \begin{proof}
We omit the proof, which is the same as   that of Lemma \ref{liple3}. 
\end{proof}

\begin{proof}[\textbf{Proof of Theorem \ref{th51}}]
With Theorem \ref{apth2}, Lemmas \ref{bll2} and \ref{bll3} at our disposal, 
Theorem \ref{th51} follows from  Lemma \ref{liple4} in the same manner as  in the case of
Theorem \ref{th41}.
We omit the details.
\end{proof}

%\section*{Acknowledgements}

\bibliographystyle{amsplain}

\bibliography{mybib}
%\bibliography{Niu_Shen_Xu_2019.bbl}

\vspace{0.5cm}

\noindent Weisheng Niu \\
School of Mathematical Science, Anhui University,
Hefei, 230601, CHINA\\
E-mail:weisheng.niu@gmail.com\\

\noindent Zhongwei Shen\\
Department of Mathematics, University of Kentucky,
Lexington, Kentucky 40506, USA.\\
E-mail: zshen2@uky.edu\\

\noindent Yao Xu \\
Institute of Mathematics, Academy of Mathematics and Systems Science, Chinese Academy of Sciences,
Beijing, 100190, CHINA\\
E-mail: xuyao89@gmail.com\\

\noindent\today

\end{document}